\documentclass[12pt, reqno]{amsart}
\usepackage{amsmath, amsthm, amscd, amsfonts, amssymb, graphicx, color}
\usepackage[bookmarksnumbered, colorlinks, plainpages]{hyperref}

\usepackage[all]{xy}

\newtheorem{theorem}{Theorem}[section]
\newtheorem{lemma}[theorem]{Lemma}
\newtheorem{proposition}[theorem]{Proposition}
\newtheorem{corollary}[theorem]{Corollary}
\theoremstyle{definition}
\newtheorem{definition}[theorem]{Definition}
\newtheorem{example}[theorem]{Example}

\theoremstyle{remark}

\numberwithin{equation}{section}

\begin{document}
\setcounter{page}{1}

\title[Strongly Unitary Equivalence and Approximately Unitary Equivalence ]%
{Strongly Unitary Equivalence and Approximately Unitary Equivalence of Normal Compact Operators over Topological Spaces}

\author[Zhu Jingming]{Zhu Jingming$^1$}

\address{$^{1}$College of Mathematics Physics and Information Engineering, Jiaxing University, Jiaxing , 314001, P.R.China;}

\email{\textcolor[rgb]{0.00,0.00,0.84}{jingmingzhu\underline{\phantom{a}}math@163.com}}




\date{Received: xxxxxx; Revised: yyyyyy; Accepted: zzzzzz.
}

\begin{abstract}
Let $A$ and $B$ be compact operators over a topological space $X$ and suppose that these operators are normal and have same distinct eigenvalues at each point. By obstruction theory, we establish a necessary and sufficient condition for $A$ and $B$ to be strongly unitarily equivalent. When $X=S^1$, we also give a sufficient condition for $A$ and $B$ to be approximately unitarily equivalent with some assumption on their eigenvalues.
\end{abstract}
\maketitle

\section{Introduction}\
It is well known that every normal matrix with complex entries is diagonalizable. An immediate consequence is that a normal matrix over $\mathbb{C}$ is determined up to unitary equivalence by its eigenvalues, counting multiplicities. R. Kadison \cite{Kadison} generalized this fact and gave an example of a normal element of $M_2(C(S^4))$ that is not diagonalizable. In \cite{Grove}, K. Grove and G. K. Pedersen considered diagonalizability of matrices over compact
Hausdorff spaces more generally. In that paper, they determined which compact Hausdorff spaces $X$ have the property that every normal matrix over $X$ is diagonalizable. Such topological spaces $X$ are rather exotic; for example, no infinite first countable compact Hausdorff space has this property. Given this failure of diagonalizability in general, Greg Friedman and Efton Park considered unitary equivalence problems for normal matrices with multiplicity-free condition, that is all the eigenvalues are distinct. In \cite{Park}, they gave a necessary and sufficient condition for two normal multiplicity-free matrices over a topological space $X$ to be unitarily equivalent. This result, in particular, yields a bound on the numbers of possible unitary equivalence classes in terms of cohomology invariants of $X$.\

Analogous to the matrix case, it is also well-known that every normal compact operator is diagonalizable. And two normal compact operators are unitarily equivalent if and only if they have the same multiplicity function $m_T(x):\mathbb{C}\to \mathbb{C}\bigcup \{\infty\}$ defined by $m_T(x)=dim \ker(T-xI)$ for a compact operator $T$. It would be interesting to investigate the unitary equivalence problem for normal compact operators over topological spaces. More precisely, suppose $X$ is a topological space. Let $C(X)$ be the $\mathbb{C}$-algebra
of complex-valued continuous functions on $X$ and let $\mathbb{B}(C(X))$ be the ring of bounded operators with entries in $C(X)$ and $\mathbb{K}(C(X))$ be the subring of compact operators i.e. the norm limit of matrices with entries in $C(X)$. By a slight abuse of terminology, we will refer to elements
of $\mathbb{K}(C(X))$ as compact operators over $X$. We call $A\in\mathbb{B}(C(X))$ to be normal/selfadjoint/unitary if $A(x)$ is normal/selfadjoint/unitary pointwise.
We can also consider a family of unitaries $\{U(x)\}_{x\in X}$ in $B(H)$ s.t. $\forall v\in H$, $U(x)v$ is continuous for $x\in X$. By a slight abuse of terminology, we call this family $\{U(x)\}_{x\in X}$ as a strongly continuous unitary operator over $X$ and we denote them as $\mathbb{U}(SC(X))$. Then two operators $A, B\in  \mathbb{K}(C(X))$ are unitarily equivalent (strongly unitarily equivalent) if there exists a $U \in \mathbb{U}(C(X))$ ($\mathbb{U}(SC(X))$)  such that $B = U^*AU$ i.e. $B(x) = U^*(x)A(x)U(x)$ for all $x \in X$. And two operators $A, B\in  \mathbb{K}(C(X))$ are approximately unitarily equivalent if $\forall \epsilon>0$, there exists a norm-continuous unitary $U \in I+\mathbb{K}(C(X))$ such that $||B(x) - U^*(x)A(x)U(x)||<\epsilon$ for all $x \in X$. One can then ask the following
question:\

$\mathbf{Question.}$ Given a topological space $X$, for two normal compact operators $A$ and $B$ over $X$, when they are strongly unitarily equivalent and when they are approximately unitarily equivalent?\

Our approach to these questions utilizes the obstruction theory. For strongly unitary equivalence, following Greg Friedman and Efton Park's paper \cite{Park}, we only need to modify their argument to adjust it to the infinite dimensional case with strong operator topology. But for approximately unitary equivalence,
because now the problem is related to dimension, we can not follow Greg Friedman and Efton Park's approach directly. And since the general problem of approximately unitary equivalence of compact operators over topological space $X$(even for cell spaces) is too hard, we restrict ourselves to the case when $X$ is one-dimensional. The case of a segment is simple, so the general one-dimensional case reduces to that for a circle. For $X=S^1$, we also construct a fiber bundle with finite dimensional fiber that encodes approximately unitary equivalence information for compact operators over $S^1$. Restricting on the compact operators with the condition that every eigenvalue at each point can be extended to be a continuous function on $[0,1]$, we then construct two finite rank normal compact operators $A_n$ and $B_n$ close to $A$ and $B$ respectively and we associate to $A_n$ and $B_n$ a continuous map from $S^1$ to the base space of the fiber bundle. We show that if this map lifts to the total space, then $A$ and $B$ are approximately unitarily equivalent. \

This article is organized as follows: In section 2, we give a sufficient and necessary condition for two normal multiplicity-free elements $A$ and $B$ in $\mathbb{K}(C(X))$ to be strongly unitarily equivalent. In section 3, we restrict our attention the normal multiplicity-free compact operators $A$ and $B$ over $S^1$ with condition (1) and we give a sufficient condition for $A$ and $B$ to be approximately unitarily equivalent. Moreover, for two normal multiplicity-free elements $A$ and $B$ in $\mathbb{K}(C(S^1))$ with $Eig(A(x))=Eig(B(x))\subset \mathbb{C}\setminus\{0\}$, we also prove that if $A$ satisfied condition (1), so does $B$ and hence $A$ and $B$ are approximately unitarily equivalent.\

\section{Strongly unitary equivalence of  normal compact operator over topological spaces}\
\subsection{A useful fiber bundle}
We construct a fiber bundle $p: E\rightarrow C$, starting with the base space. Let $\mathcal{P},\mathcal{Q}$ be the sets of 1-dimensional pairwise orthogonal projections in $B(H)$ s.t.
\[s-\lim_{F\subset_{finite}\mathcal{P}}\sum_{p\in F}p=I \text{ and }s-\lim_{F\subset_{finite}\mathcal{Q}}\sum_{q\in F}q=I\]
where $F$ is a finite subset in $\mathcal{P}$ and the limit above is the strong limit in $B(H)$. Set
\[C=\{(\mathcal{P},\mathcal{Q},\sigma):\sigma \text{ is a bijection from } \mathcal{P} \text{ to }\mathcal{Q}\}\]

We would construct a metric on $C$. Let $||\cdot||$ be the (operator) norm in $B(H)$. For $(\mathcal{P},\mathcal{Q},\sigma),(\tilde{\mathcal{P}},\tilde{\mathcal{Q}},\tilde{\sigma})\in C$, define
\begin{multline*}
d((\mathcal{P},\mathcal{Q},\sigma),(\tilde{\mathcal{P}},\tilde{\mathcal{Q}},\tilde{\sigma}))= \\
\inf\{\sup\{||p-\tau(p)||,||\sigma(p)-\tilde{\sigma}\tau(p)||:p\in\mathcal{P}\}:\tau \text{ is a bijection from } \mathcal{P} \text{ to }\widetilde{\mathcal{P}}\}
\end{multline*}

For the definition, since $||p-\tau(p)||\leq ||p||+||\tau(p)||\leq 2$ and also $||\sigma(p)-\tilde{\sigma}\tau(p)||\leq 2$, $\sup\{||p-\tau(p)||,||\sigma(p)-\tilde{\sigma}\tau(p)||:p\in\mathcal{P}\}\leq 2$. So $$d((\mathcal{P},\mathcal{Q},\sigma),(\tilde{\mathcal{P}},\tilde{\mathcal{Q}},\tilde{\sigma}))\subset [0,2]$$

\begin{lemma}
With the definition above, $d$ is a metric on $C$.
\end{lemma}
\begin{proof}
Clearly $d((\mathcal{P},\mathcal{Q},\sigma),(\tilde{\mathcal{P}},\tilde{\mathcal{Q}},\tilde{\sigma}))$ is always nonnegative. \

If $d((\mathcal{P},\mathcal{Q},\sigma),(\tilde{\mathcal{P}},\tilde{\mathcal{Q}},\tilde{\sigma}))=0$, then there is a sequence of bijections $\tau_n$ from $\mathcal{P}$ to $\tilde{\mathcal{P}}$ s.t.

\begin{equation}
||p-\tau_n(p)||<\frac{1}{n},~\forall n\in\mathbb{N},~ \forall p\in\mathcal{P}
\end{equation}
\begin{equation}
||\sigma(p)-\tilde{\sigma}\tau_n(p)||<\frac{1}{n},~\forall n\in\mathbb{N},~ \forall p\in\mathcal{P}
\end{equation}
By (2.1), for $n,m>2$, we have $||\tau_n(p)-p||<\frac{1}{n}$, $||\tau_m(p)-p||<\frac{1}{m}$ and hence $||\tau_n(p)-\tau_m(p)||<\frac{m+n}{mn}<1$, $\forall p\in\mathcal{P}$. If $\tau_n(p)\neq \tau_m(p)$ in $\tilde{\mathcal{P}}$, then $\tau_n(p)$ and $\tau_m(p)$ are orthogonal and hence $||\tau_m(p)-\tau_n(p)||\geq 1$, which is a contradiction. So $\tau_3(p)=\tau_n(p)$, $\forall n>2$, $\forall p\in\mathcal{P}$ and we have $||p-\tau_3(p)||=||p-\tau_n(p)||<\frac{1}{n}$, which induces $p=\tau_3(p)$, $\forall p\in \mathcal{P}$. Thus $\mathcal{P}=\tilde{\mathcal{P}}$ and $\tau_n=\tau_3=id$, $\forall n>2$.
By (2.2), $\sigma=\tilde{\sigma}$ and hence $\mathcal{Q}=\sigma(\mathcal{P})=\sigma(\tilde{\mathcal{P}})=\tilde{\sigma}(\tilde{\mathcal{P}})=\tilde{\mathcal{Q}}$.\

Next let $(\mathcal{P},\mathcal{Q},\sigma),(\tilde{\mathcal{P}},\tilde{\mathcal{Q}},\tilde{\sigma})\in C$. For any $n$, there exists a bijection $\tau_n:\mathcal{P}\to \tilde{\mathcal{P}}$ s.t.
\begin{align*}
&d((\mathcal{P},\mathcal{Q},\sigma),(\tilde{\mathcal{P}},\tilde{\mathcal{Q}},\tilde{\sigma}))\\
&\geq \sup\{||p-\tau_n(p)||,||\sigma(p)-\tilde{\sigma}\tau_n(p)||:p\in\mathcal{P}\}-\frac{1}{n}\\
&\geq\sup\{||\tau_n^{-1}(\tilde{p})-\tilde{p}||,||\sigma\tau_n^{-1}(\tilde{p})-\tilde{\sigma}\tilde{p}||:\tilde{p}\in\tilde{\mathcal{P}}\}-\frac{1}{n}\\
&\geq d((\tilde{\mathcal{P}},\tilde{\mathcal{Q}},\tilde{\sigma}),(\mathcal{P},\mathcal{Q},\sigma))-\frac{1}{n}
\end{align*}
So $d((\mathcal{P},\mathcal{Q},\sigma),(\tilde{\mathcal{P}},\tilde{\mathcal{Q}},\tilde{\sigma}))\geq d((\tilde{\mathcal{P}},\tilde{\mathcal{Q}},\tilde{\sigma}),(\mathcal{P},\mathcal{Q},\sigma))$ and the inverse in the same.\

Next let $(\mathcal{P},\mathcal{Q},\sigma),(\tilde{\mathcal{P}},\tilde{\mathcal{Q}},\tilde{\sigma}),(\hat{\mathcal{P}},\hat{\mathcal{Q}},\hat{\sigma})\in C$. There exist bijections $\tau_n:\mathcal{P}\to \tilde{\mathcal{P}}$ and $\nu_n:\tilde{\mathcal{P}}\to\hat{\mathcal{P}}$ s.t.
\begin{align*}
&d((\mathcal{P},\mathcal{Q},\sigma),(\tilde{\mathcal{P}},\tilde{\mathcal{Q}},\tilde{\sigma}))+d((\tilde{\mathcal{P}},\tilde{\mathcal{Q}},\tilde{\sigma}),(\hat{\mathcal{P}},\hat{\mathcal{Q}},\hat{\sigma}))\\
&\geq \sup\{||p-\tau_n(p)||,||\sigma(p)-\tilde{\sigma}\tau_n(p)||:p\in\mathcal{P}\}+\\
&\sup\{||\tilde{p}-\nu_n(\tilde{p})||,||\tilde{\sigma}(\tilde{p})-\hat{\sigma}\nu_n(\tilde{p})||:\tilde{p}\in \tilde{\mathcal{P}}\}-\frac{2}{n}\\
&\geq\sup\{||p-\tau_n(p)||+||\tau_n(p)-\nu_n\tau_n(p)||,||\sigma(p)-\tilde{\sigma}\tau_n(p)||\\
&+||\tilde{\sigma}\tau_n(p)-\hat{\sigma}\nu_n\tau_n(p)||:p\in\mathcal{P}\}-\frac{2}{n}\\
&\geq\sup\{||p-\nu_n\tau_n(p)||,||\sigma(p)-\hat{\sigma}\nu_n\tau_n(p)||:p\in\mathcal{P}\}-\frac{2}{n}\\
&\geq d((\mathcal{P},\mathcal{Q},\sigma),(\hat{\mathcal{P}},\hat{\mathcal{Q}},\hat{\sigma}))-\frac{2}{n}, \forall n\in \mathbb{N}
\end{align*}

So triangle inequality holds and $d$ is a metric on $C$.
\end{proof}

Endow $C$ with the topology induced by the metric $d$.

\begin{lemma}
Let $(\mathcal{P},\mathcal{Q},\sigma),(\tilde{\mathcal{P}},\tilde{\mathcal{Q}},\tilde{\sigma})\in C$. Suppose there exists a bijection $\tilde{\tau}:\mathcal{P}\to \tilde{\mathcal{P}}$ with
$$\sup\{||p-\tilde{\tau}(p)||,||\sigma(p)-\tilde{\sigma}\tilde{\tau}(p)||:p\in\mathcal{P}\}<\frac{1}{2}$$

Then

$$d((\hat{\mathcal{P}},\hat{\mathcal{Q}},\hat{\sigma}),(\tilde{\mathcal{P}},\tilde{\mathcal{Q}},\tilde{\sigma}))=
\sup\{||p-\tilde{\tau}(p)||,||\sigma(p)-\tilde{\sigma}\tilde{\tau}(p)||:p\in\mathcal{P}\}$$
\end{lemma}
\begin{proof}
For a bijection $\tilde{\tau}$ satisfying the condition, $||p-\tilde{\tau}(p)||<\frac{1}{2}$, $\forall p\in \mathcal{P}$. For any $\tilde{p}\in\tilde{\mathcal{P}}$ with $\tilde{p}\neq \tilde{\tau}(p)$, we have $ran(\tilde{p})\subset (ran\tilde{\tau}(p))^{\perp}$. So $\forall v\in ran(\tilde{p})$, $(\tilde{p}-\tilde{\tau}(p))v=\tilde{p}v-\tilde{\tau}(p)v=\tilde{p}v=v$, which induces $||\tilde{p}-\tilde{\tau}(p)||\geq 1$. Since $||p-\tilde{\tau}(p)||<\frac{1}{2}$, we have $||p-\tilde{p}||>\frac{1}{2}$. Hence any other choice, except for $\tilde{\tau}(p)$, will not be in the ball of radius of $\frac{1}{2}$ with center $p$. Similar argument works for $||\sigma(p)-\tilde{\sigma}\tilde{\tau}(p)||$ and hence $\tilde{\tau}:\mathcal{P}\to \tilde{\mathcal{P}}$ realizes the $\inf\{\sup\{||p-\tau(p)||,||\sigma(p)-\tilde{\sigma}\tau(p)||:p\in\mathcal{P}\}:\tau \text{ is a bijection from } \mathcal{P} \text{ to } \tilde{\mathcal{P}}\}$ i.e.
$$d((\mathcal{P},\mathcal{Q},\sigma),(\tilde{\mathcal{P}},\tilde{\mathcal{Q}},\tilde{\sigma}))=
\sup\{||p-\tilde{\tau}(p)||,||\sigma(p)-\tilde{\sigma}\tilde{\tau}(p)||:p\in\mathcal{P}\}$$
\end{proof}

\begin{lemma}
If $d((\mathcal{P},\mathcal{Q},\sigma),(\tilde{\mathcal{P}},\tilde{\mathcal{Q}},\tilde{\sigma}))<\frac{1}{4},~d((\mathcal{P},\mathcal{Q},\sigma),(\hat{\mathcal{P}},\hat{\mathcal{Q}},\hat{\sigma}))<\frac{1}{4}$, let $\tilde{\tau},\hat{\tau}$ be the bijections which realize the infimums of the distances respectively, then $\forall \epsilon>0$, $\forall p\in\mathcal{P}$
\begin{eqnarray}d((\hat{\mathcal{P}},\hat{\mathcal{Q}},\hat{\sigma}),(\tilde{\mathcal{P}},\tilde{\mathcal{Q}},\tilde{\sigma}))<\epsilon \Rightarrow
\begin{cases}
||\tilde{\tau}(p)-\hat{\tau}(p)||<\epsilon\cr  ||\tilde{\sigma}\tilde{\tau}(p)-\hat{\sigma}\hat{\tau}(p)||<\epsilon\end{cases}
\end{eqnarray}
\end{lemma}
\begin{proof}
By Lemma 2.2, we have
$$\sup\{||p-\tilde{\tau}(p)||,||\sigma(p)-\tilde{\sigma}\tilde{\tau}(p)||:p\in\mathcal{P}\}<\frac{1}{4}$$
and
$$\sup\{||p-\hat{\tau}(p)||,||\sigma(p)-\hat{\sigma}\hat{\tau}(p)||:p\in\mathcal{P}\}<\frac{1}{4}$$
So
\[\sup\{||\tilde{p}-\hat{\tau}\tilde{\tau}^{-1}(\tilde{p})||:\tilde{p}\in\tilde{\mathcal{P}}\}\leq \sup\{||\tilde{p}-\tilde{\tau}^{-1}(\tilde{p})||+||\tilde{\tau}^{-1}(\tilde{p})-\hat{\tau}\tilde{\tau}^{-1}(\tilde{p})||:\tilde{p}\in\tilde{\mathcal{P}}\}\]
\[\leq\sup\{||\tilde{\tau}(p)-p||+||p-\hat{\tau}(p)||:p\in\mathcal{P}\}<\frac{1}{2}\]
Similarly, we have $\sup\{||\tilde{\sigma}(\tilde{p})-\hat{\sigma}\hat{\tau}\tilde{\tau}^{-1}(\tilde{p})||:\tilde{p}\in\tilde{\mathcal{P}}\}<\frac{1}{2}$. By Lemma 2.2, $\hat{\tau}\tilde{\tau}^{-1}$ realizes the inf in the definition of $d((\tilde{\mathcal{P}},\tilde{\mathcal{Q}},\tilde{\sigma}),(\hat{\mathcal{P}},\hat{\mathcal{Q}},\hat{\sigma}))$. So
\begin{eqnarray}
\begin{cases}
\sup\{||\tilde{p}-\hat{\tau}\tilde{\tau}^{-1}(\tilde{p})||:\tilde{p}\in\tilde{\mathcal{P}}\}<\epsilon\Rightarrow||\tilde{\tau}(p)-\hat{\tau}(p)||<\epsilon,~\forall p\in\mathcal{P} \cr  \sup\{||\tilde{\sigma}(\tilde{p})-\hat{\sigma}\hat{\tau}\tilde{\tau}^{-1}(\tilde{p})||:\tilde{p}\in\tilde{\mathcal{P}}\}<\epsilon\Rightarrow||\tilde{\sigma}\tilde{\tau}(p)-\hat{\sigma}\hat{\tau}(p)||<\epsilon,~\forall p\in\mathcal{P}\end{cases}
\end{eqnarray}
\end{proof}

Endow $\mathbb{U}(H)$ with strong operator topology in $B(H)$. Define
$$E=\{((\mathcal{P},\mathcal{Q},\sigma),U)\in C\times \mathbb{U}(H):UpU^*=\sigma(p),~\forall p\in \mathcal{P}\}$$
$$\pi:E\to C\text{ by }\pi(((\mathcal{P},\mathcal{Q},\sigma),U))=(\mathcal{P},\mathcal{Q},\sigma)$$

Note that $((\mathcal{P},\mathcal{Q},\sigma),U)\in E\Leftrightarrow $ $U$ restricts to an isometric vector space isomorphism from $ran(p)$ to $ran(\sigma(p))$, $\forall p\in \mathcal{P}$.\

If $((\mathcal{P},\mathcal{Q},\sigma),U),((\mathcal{P},\mathcal{Q},\sigma),\tilde{U})\in \pi^{-1}((\mathcal{P},\mathcal{Q},\sigma))$, then $\forall p\in \mathcal{P}$, the restrictions of $U$ and $\tilde{U}$ from 1-dim vector space $ran(p)$ to $ran(\sigma(p))$ differs only by an isometry of $\mathbb{C}$. So we have

\begin{proposition}
If $((\mathcal{P},\mathcal{Q},\sigma),U)\in E$, then $((\mathcal{P},\mathcal{Q},\sigma),\tilde{U})\in \pi^{-1}((\mathcal{P},\mathcal{Q},\sigma),U))$ $\Leftrightarrow$
$$\tilde{U}=s-\lim_{F\subset_{finite}\mathcal{P}}\sum_{p\in F}\tilde{z_p}\sigma(p)Up$$
for some set of $\{\tilde{z_p}\}_{p\in \mathcal{P}}$ of complex number of module 1, in which  the limit above is the strong limit in $B(H)$.\

Furthermore, each such $\tilde{U}$ can be uniquely written in this form.
\end{proposition}
\begin{proof}
$\Leftarrow$: Since $\sigma(p)=UpU^*$, $\forall p\in\mathcal{P}$ and $\{p\}_{p\in\mathcal{P}}$ are pairwise orthogonal projections, $\forall v\in H$,
\[\tilde{U}\tilde{U}^*v=(s-\lim_{F\subset_{finite}\mathcal{P}}\sum_{p\in F}\tilde{z_p}\sigma(p)Up)(s-\lim_{F'\subset_{finite}\mathcal{P}}\sum_{p\in F'}\bar{\tilde{z}}_ppU^*\sigma(p))v\]
\[=(s-\lim_{F\subset_{finite}\mathcal{P}}\sum_{p\in F}\tilde{z}_p\bar{\tilde{z}}_p\sigma(p)UppU^*\sigma(p))v=s-\lim_{F\subset_{finite}\mathcal{P}}\sum_{p\in F}\sigma(p)v=v\]
So $\tilde{U}\tilde{U}^*=I$. A similar computation shows that $\tilde{U}^*\tilde{U}=I$ and hence $\tilde{U}$ is a unitary.\
\begin{align*}
&\tilde{U}p'=s-\lim_{F\subset_{finite}\mathcal{P}}\sum_{p\in F}\tilde{z_p}\sigma(p)Upp'=\tilde{z}_{p'}\sigma(p')Up'\\
&=\sigma(p')s-\lim_{F\subset_{finite}\mathcal{P}}\sum_{p\in F}\tilde{z}_p\sigma(p)Up=\sigma(p')\tilde{U},~~\forall p'\in\mathcal{P},~
\end{align*}
so $\tilde{U}p'=\sigma(p')\tilde{U}$ and hence $\tilde{U}p'\tilde{U}^*=\sigma(p')$, $\forall p'\in \mathcal{P}$.\

$\Rightarrow$: Suppose $((\mathcal{P},\mathcal{Q},\sigma),\hat{U})\in E$. $U$ and $\hat{U}$ both restrict to isometric vector space isomorphisms from $ran(p)$ to $ran(\sigma(p))$, $\forall p\in\mathcal{P}$. These isomorphisms are in the form of $\sigma(p)Up$ and $\sigma(p)\hat{U}p$ respectively.
The spaces $ran(p)$ and $ran(\sigma(p))$ are 1 dimensional for any $p\in\mathcal{P}$, so we have
\[\hat{z}_p\sigma(p)Up=\sigma(p)\hat{U}p, \text{ for some }\hat{z}_p\in S^1,~\forall p\in\mathcal{P}\]
Since $U$ could be written as $s-\lim_{F\subset_{finite}\mathcal{P}}\sum_{p\in F}\sigma(p)Up$, so
\[\hat{U}=s-\lim_{F\subset_{finite}\mathcal{P}}\sum_{p\in F}\hat{z}_p\sigma(p)Up\]
\end{proof}

\begin{lemma}\cite{Park}
Let $p$ and $\tilde{p}$ be projections in $B(H)$ and suppose $||p-\tilde{p}||<1$. Then $I+\tilde{p}-p$ maps $ran(p)$ isomorphically onto $ran(\tilde{p})$.
\end{lemma}

\begin{proposition}
Let $T^{\infty}$ be a topological space with Tychonoff product topology.
The map $\pi$ makes $E$ into a fiber bundle over $C$ with fiber homeomorphic to $T^{\infty}$.
\end{proposition}
\begin{proof}
Fixing any $(\mathcal{P},\mathcal{Q},\sigma)$, let
$$O=\{(\tilde{\mathcal{P}},\tilde{\mathcal{Q}},\tilde{\sigma})\in C:d((\mathcal{P},\mathcal{Q},\sigma),(\tilde{\mathcal{P}},\tilde{\mathcal{Q}},\tilde{\sigma}))<\frac{1}{4}\}$$
and by Lemma 2.2, let $\tilde{\tau}:\mathcal{P}\to\tilde{\mathcal{P}}$ be the bijection which realizes the distance, i.e.
\[d((\mathcal{P},\mathcal{Q},\sigma),(\tilde{\mathcal{P}},\tilde{\mathcal{Q}},\tilde{\sigma}))=\sup\{||p-\tilde{\tau}(p)||,||\sigma(p)-\tilde{\sigma}\tilde{\tau}(p)||:p\in\mathcal{P}\}<\frac{1}{4}\]

Since $||p-\tilde{\tau}(p)||<1$, by Lemma 2.5, $I+\tilde{\tau}(p)-p$ maps $ran(p)$ isometrically onto $ran(\tilde{\tau}(p))$. So for any $v_p(\neq 0)\in ran(p)$, $(I+\tilde{\tau}(p)-p)v_p(\neq 0)\in ran(\tilde{\tau}(p))$. Again since $||\sigma(p)-\tilde{\sigma}\tilde{\tau}(p)||<1$, by Lemma 2.5, $I+\tilde{\sigma}\tilde{\tau}(p)-\sigma(p)$ maps $ran(\sigma(p))$ isometrically onto $ran(\tilde{\sigma}\tilde{\tau}(p))$. So for any $w_p(\neq 0)\in ran(\sigma(p))$, $(I+\tilde{\sigma}\tilde{\tau}(p)-\sigma(p))w_p(\neq 0)\in ran(\tilde{\sigma}\tilde{\tau}(p))$. For any $((\tilde{\mathcal{P}},\tilde{\mathcal{Q}},\tilde{\sigma}),\tilde{U})\in E$ with $(\tilde{\mathcal{P}},\tilde{\mathcal{Q}},\tilde{\sigma})\in O$, denote
\[z_{\tilde{\tau},p}=\langle \tilde{U}(\frac{(I+\tilde{\tau}(p)-p)v_p}{||(I+\tilde{\tau}(p)-p)v_p||}),\frac{(I+\tilde{\sigma}\tilde{\tau}(p)-\sigma(p))w_p}{||(I+\tilde{\sigma}\tilde{\tau}(p)-\sigma(p))w_p||}\rangle\]
and we have $|z_{\tilde{\tau},p}|=1$.\

Write $T^{\infty}=\prod_{p\in\mathcal{P}}S^1$ and define $\phi:\pi^{-1}(O)\to O\times \prod_{p\in\mathcal{P}}S^1$ by
\[\phi((\tilde{\mathcal{P}},\tilde{\mathcal{Q}},\tilde{\sigma}),\tilde{U})=((\tilde{\mathcal{P}},\tilde{\mathcal{Q}},\tilde{\sigma}),\prod_{p\in\mathcal{P}}z_{\tilde{\tau},p})\]
with the product topology of metric topology and strong operator topology on $\pi^{-1}(O)$ and the product topology of metric topology and Tychonoff product topology on $O\times \prod_{p\in\mathcal{P}}S^1$.\

$\mathbf{Claim~1:}$ $\phi$ is continuous.\

For a net $\{((\tilde{\mathcal{P}}_{\alpha},\tilde{\mathcal{Q}}_{\alpha},\tilde{\sigma}_{\alpha}),\tilde{U}_{\alpha})\}_{\alpha}$ with $((\tilde{\mathcal{P}}_{\alpha},\tilde{\mathcal{Q}}_{\alpha},\tilde{\sigma}_{\alpha}),\tilde{U}_{\alpha})\rightarrow ((\tilde{\mathcal{P}},\tilde{\mathcal{Q}},\tilde{\sigma}),\tilde{U})$, we have
\begin{eqnarray}
\begin{cases}
(\tilde{\mathcal{P}}_{\alpha},\tilde{\mathcal{Q}}_{\alpha},\tilde{\sigma}_{\alpha})\rightarrow (\tilde{\mathcal{P}},\tilde{\mathcal{Q}},\tilde{\sigma}) \text{ in metric topology } \cr
\tilde{U}_{\alpha}\rightarrow \tilde{U} \text{ in strong operator topology }
\end{cases}
\end{eqnarray}

By Lemma 2.3, the top line of (2.5) implies $\tilde{\tau}_{\alpha}(p)\to \tilde{\tau}(p)$ and $\tilde{\sigma}_{\alpha}\tilde{\tau}_{\alpha}(p)\to\tilde{\sigma}\tilde{\tau}(p)$ uniformly for $\forall p\in \mathcal{P}$, in which $\tilde{\tau}_{\alpha},\tilde{\tau}$ realize the infimums of the distance $d((\mathcal{P},\mathcal{Q},\sigma),(\tilde{\mathcal{P}}_{\alpha},\tilde{\mathcal{Q}}_{\alpha},\tilde{\sigma}_{\alpha})),
d((\mathcal{P},\mathcal{Q},\sigma),(\tilde{\mathcal{P}},\tilde{\mathcal{Q}},\tilde{\sigma}))$ respectively. So for any $\epsilon>0$, for $\alpha$ big enough,
\begin{align*}
&|z_{\tilde{\tau}_{\alpha},p}-\langle\tilde{U}_{\alpha}(\frac{(I+\tilde{\tau}(p)-p)v_p}{||(I+\tilde{\tau}(p)-p)v_p||}),\frac{(I+\tilde{\sigma}\tilde{\tau}(p)-\sigma(p))w_p}{||(I+\tilde{\sigma}\tilde{\tau}(p)-\sigma(p))w_p||}\rangle|
 \\
&=|z_{\tilde{\tau}_{\alpha},p}-\langle \tilde{U}_{\alpha}(\frac{(I+\tilde{\tau}(p)-p)v_p}{||(I+\tilde{\tau}(p)-p)v_p||}),\frac{(I+\tilde{\sigma}_{\alpha}\tilde{\tau}_{\alpha}(p)-\sigma(p))w_p}{||(I+\tilde{\sigma}_{\alpha}\tilde{\tau}_{\alpha}(p)-\sigma(p))w_p||}\rangle
\\
&+\langle \tilde{U}_{\alpha}(\frac{(I+\tilde{\tau}(p)-p)v_p}{||(I+\tilde{\tau}(p)-p)v_p||}),\frac{(I+\tilde{\sigma}_{\alpha}\tilde{\tau}_{\alpha}(p)-\sigma(p))w_p}{||(I+\tilde{\sigma}_{\alpha}\tilde{\tau}_{\alpha}(p)-\sigma(p))w_p||}\rangle
\\
&-\langle\tilde{U}_{\alpha}(\frac{(I+\tilde{\tau}(p)-p)v_p}{||(I+\tilde{\tau}(p)-p)v_p||}),\frac{(I+\tilde{\sigma}\tilde{\tau}(p)-\sigma(p))w_p}{||(I+\tilde{\sigma}\tilde{\tau}(p)-\sigma(p))w_p||}\rangle|
\\
&\leq ||\tilde{U}_{\alpha}||\cdot||\frac{(I+\tilde{\tau}_{\alpha}(p)-p)v_p}{||(I+\tilde{\tau}_{\alpha}(p)-p)v_p||}-\frac{(I+\tilde{\tau}(p)-p)v_p}{||(I+\tilde{\tau}(p)-p)v_p||}||
\\
&+||\tilde{U}_{\alpha}||\cdot||\frac{(I+\tilde{\sigma}_{\alpha}\tilde{\tau}_{\alpha}(p)-\sigma(p))w_p}{||(I+\tilde{\sigma}_{\alpha}\tilde{\tau}_{\alpha}(p)-\sigma(p))w_p||}-\frac{(I+\tilde{\sigma}\tilde{\tau}(p)-\sigma(p))w_p}{||(I+\tilde{\sigma}\tilde{\tau}(p)-\sigma(p))w_p||}||\leq 2\epsilon
\end{align*}

By the bottom line of (2.5), for $\alpha$ big enough,
\begin{multline*}
||\langle\tilde{U}_{\alpha}(\frac{(I+\tilde{\tau}(p)-p)v_p}{||(I+\tilde{\tau}(p)-p)v_p||}),\frac{(I+\tilde{\sigma}\tilde{\tau}(p)-\sigma(p))w_p}{||(I+\tilde{\sigma}\tilde{\tau}(p)-\sigma(p))w_p||}\rangle\\
- \langle\tilde{U}(\frac{(I+\tilde{\tau}(p)-p)v_p}{||(I+\tilde{\tau}(p)-p)v_p||}),\frac{(I+\tilde{\sigma}\tilde{\tau}(p)-\sigma(p))w_p}{||(I+\tilde{\sigma}\tilde{\tau}(p)-\sigma(p))w_p||}\rangle
||\leq \epsilon
\end{multline*}
in which $$\langle\tilde{U}(\frac{(I+\tilde{\tau}(p)-p)v_p}{||(I+\tilde{\tau}(p)-p)v_p||}),\frac{(I+\tilde{\sigma}\tilde{\tau}(p)-\sigma(p))w_p}{||(I+\tilde{\sigma}\tilde{\tau}(p)-\sigma(p))w_p||}\rangle
=z_{\tau,p}$$

So for $\alpha$ big enough, $|z_{\tilde{\tau}_{\alpha},p}-z_{\tilde{\tau},p}|\leq 3\epsilon$. \

Next, we need to define a map $\psi:O\times \prod^{\infty} S^1\to \pi^{-1}(O)$.\

 Take $((\tilde{\mathcal{P}},\tilde{\mathcal{Q}},\tilde{\sigma}),\prod_{p\in\mathcal{P}}\zeta_{\tilde{\tau}(p)})\in O\times \prod^{\infty}S^1$ and for $\forall p\in \mathcal{P}$ and $\forall v_p,w_p$ as above, by the definition of $\mathcal{P}$ and $\mathcal{Q}$, we can span the sets
\[\{\frac{(I+\tilde{\tau}(p)-p)v_p}{||(I+\tilde{\tau}(p)-p)v_p||}:p\in\mathcal{P}\},\{\frac{(I+\tilde{\sigma}\tilde{\tau}(p)-\sigma(p))w_p}{||(I+\tilde{\sigma}\tilde{\tau}(p)-\sigma(p))w_p||}:p\in\mathcal{P}\}\]
into a same Hilbert space $H$. So there is a unitary operator $\tilde{U}$ defined by
\[
\tilde{U}(\frac{(I+\tilde{\tau}(p)-p)v_p}{||(I+\tilde{\tau}(p)-p)v_p||})
=\zeta_{\tilde{\tau}(p)}\frac{(I+\tilde{\sigma}\tilde{\tau}(p)-\sigma(p))w_p}{||(I+\tilde{\sigma}\tilde{\tau}(p)-\sigma(p))w_p||}
,~~\forall p\in \mathcal{P}
\]
Therefore $\tilde{U}$ maps $ran(\tilde{\tau}(p))$ onto $ran(\tilde{\sigma}(\tilde{\tau}(p)))$, $\forall p\in \mathcal{P}$ and hence $\tilde{U}\tilde{p}\tilde{U}^*=\tilde{\sigma}(\tilde{p})$. Thus for $((\tilde{\mathcal{P}},\tilde{\mathcal{Q}},\tilde{\sigma}),\tilde{U})\in \pi^{-1}(O)$, we can define
\[\psi((\tilde{\mathcal{P}},\tilde{\mathcal{Q}},\tilde{\sigma}),\prod_{p\in \mathcal{P}}\zeta_{\tilde{\tau}(p)})
=((\tilde{\mathcal{P}},\tilde{\mathcal{Q}},\tilde{\sigma}),\tilde{U})\]

$\mathbf{Claim~2:}$ $\psi$ is continuous.\

Assuming $((\tilde{\mathcal{P}}_{\alpha},\tilde{\mathcal{Q}}_{\alpha},\tilde{\sigma}_{\alpha}),\prod_{p\in \mathcal{P}}\zeta_{\tilde{\tau}(p)}^{\alpha})\rightarrow ((\tilde{\mathcal{P}},\tilde{\mathcal{Q}},\tilde{\sigma}),\prod_{p\in \mathcal{P}}\zeta_{\tilde{\tau}(p)})$, so
\begin{eqnarray}
\begin{cases}
(\tilde{\mathcal{P}}_{\alpha},\tilde{\mathcal{Q}}_{\alpha},\tilde{\sigma}_{\alpha})\rightarrow (\tilde{\mathcal{P}},\tilde{\mathcal{Q}},\tilde{\sigma}) \text{ in metric topology } \cr
\zeta_{\tilde{\tau}(p)}^{\alpha}\rightarrow \zeta_{\tilde{\tau}(p)} ~~~,\forall p\in \mathcal{P}
\end{cases}
\end{eqnarray}

Again by Lemma 2.3, the top line of (2.6) implies $\tilde{\tau}_{\alpha}(p)\to \tilde{\tau}(p)$ and $\tilde{\sigma}_{\alpha}\tilde{\tau}_{\alpha}(p)\to\sigma\tau(p)$ uniformly for any $p\in \mathcal{P}$, in which $\tilde{\tau}_{\alpha},\tilde{\tau}$ realize the infimums of the distance $d((\mathcal{P},\mathcal{Q},\sigma),(\tilde{\mathcal{P}}_{\alpha},\tilde{\mathcal{Q}}_{\alpha},\tilde{\sigma}_{\alpha}))
,d((\mathcal{P},\mathcal{Q},\sigma),(\tilde{\mathcal{P}},\tilde{\mathcal{Q}},\tilde{\sigma}))$ respectively.\

Denote $$\psi((\tilde{\mathcal{P}}_{\alpha},\tilde{\mathcal{Q}}_{\alpha},\tilde{\sigma}_{\alpha}),\prod_{p\in \mathcal{P}}\zeta_{\tilde{\tau}(p)}^{\alpha})=((\tilde{\mathcal{P}}_{\alpha},\tilde{\mathcal{Q}}_{\alpha},\tilde{\sigma}_{\alpha}),\tilde{U}_{\alpha})$$
in which $\tilde{U}_{\alpha}$ is defined by
\[
\tilde{U}_{\alpha}(\frac{(I+\tilde{\tau}_{\alpha}(p)-p)v_p}{||(I+\tilde{\tau}_{\alpha}(p)-p)v_p||})
=\zeta_{\tilde{\tau}(p)}^{\alpha}(\frac{(I+\tilde{\sigma}_{\alpha}\tilde{\tau}_{\alpha}(p)-\sigma(p))w_p}{||(I+\tilde{\sigma}_{\alpha}\tilde{\tau}_{\alpha}(p)-\sigma(p))w_p||})
,~~\forall p\in \mathcal{P}\]

So we have, for any $\epsilon>0$, for $\alpha$ big enough,
\[
||\frac{(I+\tilde{\tau}_{\alpha}(p)-p)v_p}{||(I+\tilde{\tau}_{\alpha}(p)-p)v_p||}
-\frac{(I+\tilde{\tau}(p)-p)v_p}{||(I+\tilde{\tau}(p)-p)v_p||}||<\epsilon
\]
and hence
\[
||\tilde{U}_{\alpha}\frac{(I+\tilde{\tau}_{\alpha}(p)-p)v_p}{||(I+\tilde{\tau}_{\alpha}(p)-p)v_p||}
-\tilde{U}_{\alpha}\frac{(I+\tilde{\tau}(p)-p)v_p}{||(I+\tilde{\tau}(p)-p)v_p||}||<\epsilon
\]
for $||\tilde{U}_{\alpha}||=1$.
Similarly, for $\alpha$ big enough,
\[
||\frac{(I+\tilde{\sigma}_{\alpha}\tilde{\tau}_{\alpha}(p)-\sigma(p))w_p}{||(I+\tilde{\sigma}_{\alpha}\tilde{\tau}_{\alpha}(p)-\sigma(p))w_p||}
-\frac{(I+\tilde{\sigma}\tilde{\tau}(p)-\sigma(p))w_p}{||(I+\tilde{\sigma}\tilde{\tau}(p)-\sigma(p))w_p||}||
<\epsilon\]

Since $\zeta_{\tilde{\tau}(p)}^{\alpha}\rightarrow \zeta_{\tilde{\tau}(p)} ,\forall p\in \mathcal{P}$, for $\alpha$ big enough,
\[||\zeta_{\tilde{\tau}(p)}^{\alpha}\frac{(I+\tilde{\sigma}\tilde{\tau}(p)-\sigma(p))w_{p}}{||(I+\tilde{\sigma}\tilde{\tau}(p)-\sigma(p))w_{p}||}
-\zeta_{\tilde{\tau}(p)}\frac{(I+\tilde{\sigma}\tilde{\tau}(p)-\sigma(p))w_{p}}{||(I+\tilde{\sigma}\tilde{\tau}(p)-\sigma(p))w_{p}||}||<\epsilon\]
in which
\[\zeta_{\tilde{\tau}(p)}\frac{(I+\tilde{\sigma}\tilde{\tau}(p)-\sigma(p))w_{p}}{||(I+\tilde{\sigma}\tilde{\tau}(p)-\sigma(p))w_{p}||}
=\tilde{U}(\frac{(I+\tilde{\tau}(p)-p)v_p}{||(I+\tilde{\tau}(p)-p)v_p||})
\]

So $\forall p\in\mathcal{P}$, for $\alpha$ big enough
\begin{multline*}
||\tilde{U}_{\alpha}(\frac{(I+\tilde{\tau}(p)-p)v_p}{||(I+\tilde{\tau}(p)-p)v_p||}) -\tilde{U}(\frac{(I+\tilde{\tau}(p)-p)v_p}{||(I+\tilde{\tau}(p)-p)v_p||})||\\
\leq||\tilde{U}_{\alpha}(\frac{(I+\tilde{\tau}(p)-p)v_p}{||(I+\tilde{\tau}(p)-p)v_p||}) -\tilde{U}_{\alpha}(\frac{(I+\tilde{\tau}_{\alpha}(p)-p)v_p}{||(I+\tilde{\tau}_{\alpha}(p)-p)v_p||}) ||\\
+||\zeta_{\tilde{\tau}(p)}^{\alpha}\frac{(I+\tilde{\sigma}_{\alpha}\tilde{\tau}_{\alpha}(p)-\sigma(p))w_{p}}{||(I+\tilde{\sigma}_{\alpha}\tilde{\tau}_{\alpha}(p)-\sigma(p))w_{p}||}-
\tilde{U}(\frac{(I+\tilde{\tau}(p)-p)v_p}{||(I+\tilde{\tau}(p)-p)v_p||})||\leq 2\epsilon
\end{multline*}
and hence $\tilde{U}_{\alpha}\to \tilde{U}$ in strong operator topology.\

It is easy to see the map $\phi$ and $\psi$ are inverses of each other and hence $\phi$ is a homeomorphism. Therefore $E$ is a fiber bundle over $C$.
\end{proof}

\subsection{Back to operators}

For a compact Hausdorff space $X$, let $A\in \mathbb{K}(C(X))$. We say $A$ is multiplicity-free if any $\lambda(x)$ in the eigenvalue set $Eig(A(x))$  has multiplicity one.
Let $A,B$ be two normal multiplicity-free elements in $\mathbb{K}(C(X))$ with $Eig(A(x))=Eig(B(x))$, $\forall x\in X$. For any fixed $x\in X$, given an element $\lambda(x)\in Eig(A(x))$, we can associate to $\lambda(x)$ the eigenprojection $P(x)_{\lambda(x)}$ of $A(x)$. Similarly, we can associate to $\lambda(x)$ the eigenprojection $Q(x)_{\lambda(x)}$ of $B(x)$. Thus, $\forall x\in X$, let $\mathcal{P}$ be the eigenprojection set of $A(x)$ and $\mathcal{Q}$ be the eigenprojection set of $B(x)$, then we can define a bijection $\sigma$ from $\mathcal{P}$ to $\mathcal{Q}$ s.t. $\sigma(P(x)_{\lambda(x)})=Q(x)_{\lambda(x)}$, $\forall \lambda(x)\in Eig(A(x))=Eig(B(x))$, $\forall x\in X$, which determines an element in $C$. Since the eigenprojections of $A(x)$ and $B(x)$ vary continuously, we can assign to the pair $(A,B)$ a continuous map $\Phi_{A,B}:X\to C$.

\begin{theorem}
Let $A,B$ be two normal multiplicity-free elements in $\mathbb{K}(C(X))$ with $Eig(A(x))=Eig(B(x))$, $\forall x\in X$, then $A,B$ are strongly unitarily equivalent $\Leftrightarrow$ the map $\Phi_{A,B}:X\to C$ lifts to a continuous map $\tilde{\Phi}_{A,B}:X\to E$.
\end{theorem}
\begin{proof}
If $UAU^*=B$ for some $U\in \mathbb{U}(SC(X))$, then $U(x)A(x)U(x)^*=B(x)$, $\forall x\in X$. So unitary $U(x)$ conjugates each eigenprojection of $A(x)$ to the corresponding eigenprojection of $B(x)$, $\forall x\in X$ i.e. $U(x)P(x)_{\lambda(x)}U(x)^*=Q(x)_{\lambda(x)}$. Therefore $(\Phi_{A,B}(x),U(x))\in E$, $\forall x\in X$ and hence the continuous map $\tilde{\Phi}_{A,B}(x)=(\Phi_{A,B}(x),U(x))$ is a continuous lifting of $\Phi_{A,B}$.\

Conversely, suppose $\pi \tilde{\Phi}_{A,B}=\Phi_{A,B}$ for some continuous map $\tilde{\Phi}_{A,B}:X\to E$. So $\forall x\in X$, $\tilde{\Phi}_{A,B}(x)=(\Phi_{A,B}(x),U(x))$, for some $U\in\mathbb{U}(SC(X))$. So for any $\lambda(x)\in Eig(A(x))=Eig(B(x))$, we have $U(x)P(x)_{\lambda(x)}U(x)^*=Q(x)_{\lambda(x)}$, thus $U(x)A(x)U(x)^*=B(x)$, $\forall x\in X$. 
\end{proof}

\section{Approximately unitary equivalence of normal compact operators over $S^1$}

\subsection{A useful bundle}

We construct a fiber bundle $p: E\rightarrow C_n$, starting with the base space. Let $\mathcal{P}_n,\mathcal{Q}_n$ be the sets of $n$ rank-one pairwise orthogonal projections in $B(H)$ s.t. $\exists$  rank-$n$ projections $I(\mathcal{P}_n)$ and $I(\mathcal{Q}_n)$ satisfies
\[\sum_{p\in\mathcal{P}_n}p=I(\mathcal{P}_n)\text{ and }I(\mathcal{Q}_n)=\sum_{q\in\mathcal{Q}_n}q\]
We do not assume any ordering of the elements in $\mathcal{P}_n$ and $\mathcal{Q}_n$.\

Set
\[C_n=\{(\mathcal{P}_n,\mathcal{Q}_n,\sigma_n):\sigma_n\text{ is a bijection from $\mathcal{P}_n$ to $\mathcal{Q}_n$}\}\]
For $(\mathcal{P}_n,\mathcal{Q}_n,\sigma_n),(\tilde{\mathcal{P}}_n,\tilde{\mathcal{Q}}_n,\tilde{\sigma}_n)\in C_n$, define
\begin{multline*}
  d((\mathcal{P}_n,\mathcal{Q}_n,\sigma_n),(\tilde{\mathcal{P}}_n,\tilde{\mathcal{Q}}_n,\tilde{\sigma}_n))= \\
  \min\{\max\{||p-\tau(p)||,||\sigma_n(p)-\tilde{\sigma}_n\tau(p)||:p\in\mathcal{P}_n\}:\tau \text{ is a bijection from } \mathcal{P}_n \text{ to }\mathcal{Q}_n\}
\end{multline*}
By Proposition 2.1 in \cite{Park}, $C_n$ is a metric space with the metric defined above. Endow the unitary operator set $\mathbb{U}(ran(I(\mathcal{P}_n)),ran((\mathcal{Q}_n)))$ with its usual topology and let $U(\mathcal{P}_n,\mathcal{Q}_n)$ be an element in $\mathbb{U}(ran(I(\mathcal{P}_n)),ran(I(\mathcal{Q}_n)))$ s.t. $$U(\mathcal{P}_n,\mathcal{Q}_n)^*U(\mathcal{P}_n,\mathcal{Q}_n)=I(\mathcal{P}_n)\text{ and } U(\mathcal{P}_n,\mathcal{Q}_n)U(\mathcal{P}_n,\mathcal{Q}_n)^*=I(\mathcal{Q}_n)$$
Define
\begin{multline*}
E_n=\{((\mathcal{P}_n,\mathcal{Q}_n,\sigma_n),U(\mathcal{P}_n,\mathcal{Q}_n))\in C_n\times \mathbb{U}(ran(I(\mathcal{P}_n)),ran(I(\mathcal{Q}_n))):\\
U(\mathcal{P}_n,\mathcal{Q}_n)pU(\mathcal{P}_n,\mathcal{Q}_n)^*=\sigma(p),~~\forall p\in\mathcal{P}_n\}
\end{multline*}
and a map $\pi: E_n\to C_n$ by
\[\pi((\mathcal{P}_n,\mathcal{Q}_n,\sigma_n),U(\mathcal{P}_n,\mathcal{Q}_n))=(\mathcal{P}_n,\mathcal{Q}_n,\sigma_n), ~~~\forall ((\mathcal{P}_n,\mathcal{Q}_n,\sigma_n),U(\mathcal{P}_n,\mathcal{Q}_n))\in E_n\]
Note that $((\mathcal{P}_n,\mathcal{Q}_n,\sigma_n),U(\mathcal{P}_n,\mathcal{Q}_n))\in E_n\Leftrightarrow U(\mathcal{P}_n,\mathcal{Q}_n)$ restricts to an isometric vector space isomorphism from $ran(p)$ to $ran(\sigma(p))$, $\forall p\in\mathcal{P}_n$.\

\begin{proposition}
If $((\mathcal{P}_n,\mathcal{Q}_n,\sigma_n),U(\mathcal{P}_n,\mathcal{Q}_n))\in E_n$, then $$((\mathcal{P}_n,\mathcal{Q}_n,\sigma_n),\tilde{U}(\mathcal{P}_n,\mathcal{Q}_n))\in \pi^{-1}((\mathcal{P}_n,\mathcal{Q}_n,\sigma_n))$$
$$\Leftrightarrow
\tilde{U}(\mathcal{P}_n,\mathcal{Q}_n)=\sum_{p\in \mathcal{P}_n}\tilde{z_p}\sigma(p)U(\mathcal{P}_n,\mathcal{Q}_n)p$$
for some set of $\{\tilde{z_p}\}_{p\in \mathcal{P}_n}$ of complex number of module 1.\

Furthermore, each such $\tilde{U}(\mathcal{P}_n,\mathcal{Q}_n)$ can be uniquely written in this form.
\end{proposition}
\begin{proof}
The proof is similar with the proof of Proposition 2.1 in \cite{Park}.
\end{proof}

As a consequence of the proposition above, we can identify $\pi^{-1}((\mathcal{P}_n,\mathcal{Q}_n,\sigma_n))$ with $\mathbb{T}^n\simeq \prod_{p\in\mathcal{P}_n}S^1$. In fact, we will prove that $E_n$ is a $\mathbb{T}^n$-fiber bundle over $C_n$. Firstly we should introduce a technique lemma similar with Lemma 2.4 in \cite{Park}.

\begin{lemma}
Let $p$ and $\tilde{p}$ be projections in $B(H)$ and suppose $||p-\tilde{p}||<1$. Then for the identity $I\in B(H)$, $I+\tilde{p}-p$ maps $ran(p)$ isomorphically onto $ran(\tilde{p})$.
\end{lemma}
\begin{proof}
Since by Theorem 1.2.2 in \cite{Murphy}, $I+\tilde{p}-p$ is invertible. The rest of proof is exactly the same as Lemma 2.4 in \cite{Park}.
\end{proof}

\begin{proposition}
The map $\pi$ makes $E_n$ into a fiber bundle over $C_n$ with fiber homeomorphic to $\mathbb{T}^{n}$ with usual topology.
\end{proposition}
\begin{proof}
Fixing any $(\mathcal{P}_n,\mathcal{Q}_n,\sigma_n)$, let
$$O_n=\{(\tilde{\mathcal{P}}_n,\tilde{\mathcal{Q}}_n,\tilde{\sigma}_n)\in C:d((\mathcal{P}_n,\mathcal{Q}_n,\sigma_n),(\tilde{\mathcal{P}}_n,\tilde{\mathcal{Q}}_n,\tilde{\sigma}_n))<\frac{1}{4}\}$$
and by Lemma 2.2 in \cite{Park}, let $\tilde{\tau}:\mathcal{P}_n\to\tilde{\mathcal{P}}_n$ be the bijection which realizes the distance, i.e.
\[d((\mathcal{P}_n,\mathcal{Q}_n,\sigma_n),(\tilde{\mathcal{P}}_n,\tilde{\mathcal{Q}}_n,\tilde{\sigma}_n))=\max\{||p-\tilde{\tau}(p)||,||\sigma_n(p)-\tilde{\sigma}_n\tilde{\tau}(p)||:p\in\mathcal{P}_n\}<\frac{1}{4}\]

Taking $(\tilde{\mathcal{P}}_n,\tilde{\mathcal{Q}}_n,\tilde{\sigma}_n)\in O_n$, since $||p-\tilde{\tau}(p)||<1$, by Lemma 3.2, $I+\tilde{\tau}(p)-p$ maps $ran(p)$ isometrically onto $ran(\tilde{\tau}(p))$. So for any $v_p(\neq 0)\in ran(p)$, $(I+\tilde{\tau}(p)-p)v_p(\neq 0)\in ran(\tilde{\tau}(p))$. Again since $||\sigma_n(p)-\tilde{\sigma}_n\tilde{\tau}(p)||<1$, by Lemma 3.2, $I+\tilde{\sigma}_n\tilde{\tau}(p)-\sigma_n(p)$ maps $ran(\sigma_n(p))$ isometrically onto $ran(\tilde{\sigma}_n\tilde{\tau}(p))$. For each $p\in\mathcal{P}_n$,  $\forall w_p(\neq 0)\in ran(\sigma_n(p))$, $(I+\tilde{\sigma}_n\tilde{\tau}(p)-\sigma_n(p))w_p(\neq 0)\in ran(\tilde{\sigma}_n\tilde{\tau}(p))$, we denote
\[z_{\tilde{\tau},p}=\langle \tilde{U}(\tilde{\mathcal{P}}_n,\tilde{\mathcal{Q}}_n)(\frac{(I+\tilde{\tau}(p)-p)v_p}{||(I+\tilde{\tau}(p)-p)v_p||}),\frac{(I+\tilde{\sigma}_n\tilde{\tau}(p)-\sigma_n(p))w_p}{||(I+\tilde{\sigma}_n\tilde{\tau}(p)-\sigma_n(p))w_p||}\rangle\]
and we have $|z_{\tilde{\tau},p}|=1$.\

Write $\mathbb{T}^{n}=\prod_{p\in\mathcal{P}_n}S^1$ and define $\phi:\pi^{-1}(O_n)\to O_n\times \prod_{p\in\mathcal{P}_n}S^1$ by
\[\phi(((\tilde{\mathcal{P}}_n,\tilde{\mathcal{Q}}_n,\tilde{\sigma}_n),\tilde{U}(\tilde{\mathcal{P}}_n,\tilde{\mathcal{Q}}_n)))=((\tilde{\mathcal{P}}_n,\tilde{\mathcal{Q}}_n,\tilde{\sigma}_n),\prod_{p\in\mathcal{P}_n}z_{\tilde{\tau},p})\]
with the product topology on $O_n\times \prod_{p\in\mathcal{P}_n}S^1$.\

To show that $\phi$ is continuous, it clearly suffices to prove that the map $$((\tilde{\mathcal{P}_n},\tilde{\mathcal{Q}}_n,\tilde{\sigma}_n),\tilde{U}(\tilde{\mathcal{P}}_n,\tilde{\mathcal{Q}}_n))\to \tilde{z}_{\tau,p}$$
is continuous for each $p\in\mathcal{P}_n$. Define $\Phi_{p} : \pi^{-1}(O_n) \to H$ by the formula
$$\Phi_{p} ((\tilde{\mathcal{P}}_n,\tilde{\mathcal{Q}}_n,\tilde{\sigma}_n),\tilde{U}(\tilde{\mathcal{P}}_n,\tilde{\mathcal{Q}}_n))=(I + \tilde{\tau}(p)-p )v_p$$

Suppose that $((\tilde{\mathcal{P}}_n,\tilde{\mathcal{Q}}_n,\tilde{\sigma}_n),\tilde{U}(\tilde{\mathcal{P}}_n,\tilde{\mathcal{Q}}_n))$ and $((\hat{\mathcal{P}}_n,\hat{\mathcal{Q}}_n,\hat{\sigma}_n),\hat{U}(\hat{\mathcal{P}}_n,\hat{\mathcal{Q}}_n))$ are in $\pi^{-1}(O_n)$ and $$d((\tilde{\mathcal{P}}_n,\tilde{\mathcal{Q}}_n,\tilde{\sigma}_n),(\hat{\mathcal{P}}_n,\hat{\mathcal{Q}}_n,\hat{\sigma}_n))<\epsilon$$
By the definition of the distance, we obtain
\begin{multline*}
||\Phi_p((\tilde{\mathcal{P}}_n,\tilde{\mathcal{Q}}_n,\tilde{\sigma}_n),\tilde{U}(\tilde{\mathcal{P}}_n,\tilde{\mathcal{Q}}_n))-\Phi_p((\hat{\mathcal{P}}_n,\hat{\mathcal{Q}}_n,\hat{\sigma}_n),\hat{U}(\hat{\mathcal{P}}_n,\hat{\mathcal{Q}}_n))||\\
=||(I+\tilde{\tau}(p)-p)v_p-(I+\hat{\tau}(p)-p)v_p||=||(\tilde{\tau}(p)-\hat{\tau}(p))v_p||\leq ||\tilde{\tau}(p)-\hat{\tau}(p)||<\epsilon
\end{multline*}
in which $\hat{\tau}:\mathcal{P}_n\to\hat{\mathcal{P}}_n$ is the bijection that realizes the minimum of the distance $d((\mathcal{P}_n,\mathcal{Q}_n,\sigma_n),(\hat{\mathcal{P}}_n,\hat{\mathcal{Q}}_n,\hat{\sigma}_n))$.
Therefore $\Phi_p$ is continuous, $\forall p\in\mathcal{P}_n$ and by the similar argument the map $\Psi_{p} : \pi^{-1}(O_n) \to\mathbb{C}^n$ defined by the formula
$$\Psi_{p} ((\tilde{\mathcal{P}}_n,\tilde{\mathcal{Q}}_n,\tilde{\sigma}_n),\tilde{U}(\tilde{\mathcal{P}}_n,\tilde{\mathcal{Q}}_n))=(I+\tilde{\sigma}_n\tilde{\tau}(p)-\sigma_n(p))w_p$$
is also continuous. So by the definition of $z_{\tilde{\tau},p}$, $((\tilde{\mathcal{P}}_n,\tilde{\mathcal{Q}}_n,\tilde{\sigma}_n),\tilde{U}(\tilde{\mathcal{P}}_n,\tilde{\mathcal{Q}}_n))\to \tilde{z}_{\tau,p}$ is continuous for each $p\in\mathcal{P}_n$.\

Next, we define a map $\psi:O_n\times \mathbb{T}^n\to \pi^{-1}(O_n)$. For $(\tilde{\mathcal{P}}_n,\tilde{\mathcal{Q}}_n,\tilde{\sigma}_n)$ in $\pi^{-1}(O_n)$ and $\tilde{\tau}$, $v_p$, $w_p$ as above, we span the sets
\[
\{\frac{(I+\tilde{\tau}(p)-p)v_p}{||(I+\tilde{\tau}(p)-p)v_p||}:p\in\mathcal{P}_n\}\text{ to }ran(I(\tilde{\mathcal{P}}_n))
\]
and
\[
\{\frac{(I+\tilde{\sigma}_n\tilde{\tau}(p)-\sigma_n(p))w_p}{||(I+\tilde{\sigma}_n\tilde{\tau}(p)-\sigma_n(p))w_p||}:p\in\mathcal{P}_n\}\text{ to }ran(I(\tilde{\mathcal{Q}}_n))
\]
So for $\{\zeta_{\tilde{\tau}(p)}\}_{p\in\mathcal{P}_n}\in\mathbb{T}^n$, we can define an element $\tilde{U}(\tilde{\mathcal{P}}_n,\tilde{\mathcal{Q}}_n)$ in $\mathbb{U}(ran(I(\tilde{\mathcal{P}}_n)),ran(I(\tilde{\mathcal{Q}}_n)))$ by
\[\tilde{U}(\tilde{\mathcal{P}}_n,\tilde{\mathcal{Q}}_n)(\frac{(I+\tilde{\tau}(p)-p)v_p}{||(I+\tilde{\tau}(p)-p)v_p||})
=\zeta_{\tilde{\tau}(p)}(\frac{(I+\tilde{\sigma}_n\tilde{\tau}(p)-\sigma_n(p))w_p}{||(I+\tilde{\sigma}_n\tilde{\tau}(p)-\sigma_n(p))w_p||}),\forall p\in\mathcal{P}_n
\]
s.t.
$$\tilde{U}(\tilde{\mathcal{P}}_n,\tilde{\mathcal{Q}}_n)^*\tilde{U}(\tilde{\mathcal{P}}_n,\tilde{\mathcal{Q}}_n)=I(\tilde{\mathcal{P}}_n)\text{ and } \tilde{U}(\tilde{\mathcal{P}}_n,\tilde{\mathcal{Q}}_n)\tilde{U}(\tilde{\mathcal{P}}_n,\tilde{\mathcal{Q}}_n)^*=I(\widetilde{\mathcal{Q}}_n)$$
So $\tilde{U}(\tilde{\mathcal{P}}_n,\tilde{\mathcal{Q}}_n)$ maps $ran(\tilde{p})$ onto $ran(\tilde{\sigma}_n(\tilde{p}))$, $\forall \tilde{p}\in\tilde{\mathcal{P}}_n$ and hence $$\tilde{U}(\tilde{\mathcal{P}}_n,\tilde{\mathcal{Q}}_n)\tilde{p}\tilde{U}(\tilde{\mathcal{P}}_n,\tilde{\mathcal{Q}}_n)^*=\tilde{\sigma}_n(\tilde{p})$$

Therefore $((\tilde{\mathcal{P}}_n,\tilde{\mathcal{Q}}_n,\tilde{\sigma}_n),\tilde{U}(\tilde{\mathcal{P}}_n,\tilde{\mathcal{Q}}_n))\in \pi^{-1}(O_n)$ and we can define
\[\psi(((\tilde{\mathcal{P}}_n,\tilde{\mathcal{Q}}_n,\tilde{\sigma}_n),\prod_{p\in\mathcal{P}_n}\zeta_{\tilde{\tau}(p)}))=((\tilde{\mathcal{P}}_n,\tilde{\mathcal{Q}}_n,\tilde{\sigma}_n),\tilde{U}(\tilde{\mathcal{P}}_n,\tilde{\mathcal{Q}}_n))\]

Similar with the argument in the proof of Proposition 2.6, $\psi$ is continuous. It is easy to see $\phi$ and $\psi$ are inverses to each other, therefore $\phi$ is a homeomorphism and hence $E_n$ is a fiber bundle.
\end{proof}

\subsection{Back to operators}

For two normal multiplicity-free elements $A$ and $B$ in $\mathbb{K}(C(S^1))$ with the eigenvalue sets $Eig(A(x))=Eig(B(x))\subset \mathbb{C}\setminus\{0\}$, $\forall x\in S^1$. For any $n\in \mathbb{N}$, we will give a sufficient condition for the existence of continuous operator-valued function $U\in I+\mathbb{K}(C(S^1))$ s.t.
\[
||U(x)^*A(x)U(x)-B(x)||<\frac{37}{n}
\]
 and
\[U(x)^*U(x)=I=U(x)U(x)^*,~~\forall x\in S^1
\]
with some assumption on the eigenvalue sets of $A$ and $B$.\

Let $\bar{A},\bar{B}\in \mathbb{K}(C([0,1]))$ defined by
$$\bar{A}(x)=A(e^{2\pi xi})\text{ and }\bar{B}(x)=B(e^{2\pi xi}),~~\forall x\in [0,1]$$

So we have $\bar{A}(0)=\bar{A}(1)$ and $\bar{B}(0)=\bar{B}(1)$. It is easy to see that $\bar{A},\bar{B}\in \mathbb{K}(C([0,1]))$ are still normal multiplicity-free and $Eig(\bar{A}(x))=Eig(\bar{B}(x))\subset\mathbb{C}\setminus\{0\}$, $\forall x\in[0,1]$.\

From now on, we will restrict our attention to the elements in $\mathbb{K}(C(S^1))$ satisfying the following condition.

\begin{definition}
For an element $\bar{A}\in\mathbb{K}(C[0,1])$, we say $\bar{A}$ satisfies condition (1), if for any $x\in [0,1]$, $\forall \lambda_x\in Eig(\bar{A}(x))$,
$\exists$ a continuous function $\lambda(x)\in C[0,1]$ s.t. $\lambda(x)=\lambda_x$ and $\lambda(x)\in Eig(\bar{A}(x))$, $\forall x\in [0,1]$. We say $A\in \mathbb{K}(C(S^1))$ satisfies condition (1), if $\bar{A}$ satisfies condition (1).
\end{definition}

\begin{definition}
For a sequence of continuous matrix-valued functions $\{\bar{A}_n(x)\}_n$, in which $\bar{A}_n\in M_{n\times n}(C[0,1])$, we say $\{\bar{A}_n(x)\}_n$ satisfying condition (2), if $\forall n\in \mathbb{N}$, $\exists F_n\in M_{n\times n}(C[0,1])$ s.t.\\

(I). $F_n(x)\bar{\Lambda}_n(x)F_n(x)^*=\bar{A}_n(x), ~~\forall x\in [0,1]$\\

(II). $\lambda_i^{(n)}(x)=\lambda_i^{(i)}(x),~~\forall i\leq n,~~\forall x\in[0,1]$\\

in which
$$\bar{\Lambda}_n(x)=
\left(
  \begin{array}{cccc}
   \lambda_1^{(n)}(x) & 0 & ... & 0 \\
  0 & \lambda_2^{(n)}(x) & ... & 0 \\
     ... & ... & ... & ... \\
    0 & 0 & 0 & \lambda_n^{(n)}(x) \\
  \end{array}
\right)$$
and $\lambda_i^{(n)}(x)\in C[0,1]$.\\

Let
$$Ac(\{\bar{\Lambda}_n(x)\}_n)=\{\lambda_n^{(n)}(x):n\in\mathbb{N}\},\forall x\in [0,1]$$

We say $\bar{A}\in \mathbb{K}(C[0,1])$ satisfying condition (2), if $\exists$ a sequence $\{\bar{A}_n(x)\}_n$ with condition (2) s.t.
\[
\bar{A}_n\to \bar{A}\text{ in norm as $n\to\infty$ and }Eig(\bar{A}(x))=Ac(\{\bar{\Lambda}_n(x)\}_n),~~\forall x\in [0,1]
\]
We say $A\in \mathbb{K}(C(S^1))$ satisfies condition (2), if $\bar{A}$ satisfies condition (2).
\end{definition}

For an element $\bar{A}\in \mathbb{K}(C[0,1])$ satisfying condition (2), let $\lambda_i(x)=\lambda_i^{(i)}(x)$, for any $i\in \mathbb{N}$. Since $Eig(\bar{A}(x))=Ac(\{\bar{\Lambda}_n(x)\}_n)$, for any $\lambda_x\in Eig(\bar{A}(x))$, there exists $i\in \mathbb{N}$ s.t. $\lambda_i(x)=\lambda_x$. So $\bar{A}\in \mathbb{K}(C[0,1])$ satisfies condition (1).\

If $Eig(\bar{A}(x))\subset \mathbb{C}\setminus \{0\}$, $\forall x\in [0,1]$, for an element $\bar{A}\in \mathbb{K}(C[0,1])$ satisfying condition (1), there is a sequence of continuous functions $\{\lambda_i(x)\}_i$ s.t. $\{\lambda_i(x)\}_i=Eig(\bar{A}(x))$, $\forall x\in [0,1]$. Let $A$ be normal and have  multiplicity-free condition i.e.
$$\lambda_i(x)\neq \lambda_j(x),\forall i\neq j\in \mathbb{N}\text{ and }\forall x\in[0,1]$$
Let $res(\bar{A}(x))$ be the resolvent set of $\bar{A}(x)$, so we can choose $\delta_i(x)>0$, $\forall x\in[0,1]$ s.t.
\[
B(\lambda_i(x),\delta_i(x))\bigcap B(\lambda_j(x),\delta_j(x))=\emptyset,~~\forall i\neq j\in \mathbb{N},\forall x\in[0,1]
\]
and $B(\lambda_i(x),\delta_i(x))\setminus \{\lambda_i(x)\}\subset res(\bar{A}(x))$.\

For any fixed $t\in[0,1]$, because $\overline{B(\lambda_i(t),\frac{\delta_i(x)}{2})}\setminus B(\lambda_i(t),\frac{\delta_i(x)}{4})$ is a closed subset and $(\lambda-\bar{A}(t))^{-1}$ exists  $\forall \lambda\in\overline{B(\lambda_i(t),\frac{\delta_i(x)}{2})}\setminus B(\lambda_i(t),\frac{\delta_i(x)}{4})$, we have
\[
\inf\{||\lambda-\bar{A}(t)^{-1}||:\lambda\in\overline{B(\lambda_i(t),\frac{\delta_i(x)}{2})}\setminus B(\lambda_i(t),\frac{\delta_i(x)}{4})\}=r_{i,\delta_i(x)}(t)>0
\]
So
\begin{equation*}\
 (\lambda-A(x))^{-1}=(\lambda-A(t))^{-1}(1-(A(x)-A(t))(\lambda-A(t))^{-1})^{-1}
\end{equation*}
exists for $||\bar{A}(x)-\bar{A}(t)||<r_{i,\delta_i(x)}(t)$ , $\forall \lambda\in\overline{B(\lambda_i(t),\frac{\delta_i(x)}{2})}\setminus B(\lambda_i(t),\frac{\delta_i(x)}{4})$.\

Since $\bar{A}(x)$ is continuous, $(\lambda-A(x))^{-1}$ exists for $x\in(t-\alpha_{i,\delta_i(x)}(t),t+\alpha_{i,\delta_i(x)}(t))$ for some $\alpha_{i,\delta_i(x)}(t)>0$ and $\forall \lambda\in\overline{B(\lambda_i(t),\frac{\delta_i(x)}{2})}\setminus B(\lambda_i(t),\frac{\delta_i(x)}{4})$.
 Therefore
\begin{equation*}
(\overline{B(\lambda_i(t),\frac{\delta_i(x)}{2})}\bigcap Eig(\bar{A}(x))=\{\lambda_i(x)\},\forall x\in(t-\alpha_{i,\delta_i(x)}(t),t+\alpha_{i,\delta_i(x)}(t))
\end{equation*}
and by (4.1) in \cite{Baumgartel} and the continuity of $\bar{A}(t)$, for $t\in [0,1]$, we have
\begin{lemma}
If $Eig(\bar{A}(x))\subset \mathbb{C}\setminus \{0\}$, for any fixed $x\in [0,1]$, let $\lambda_i(x)\in Eig(\bar{A}(x))$ and $\delta_i(x)$ as above, then the Riesz projection
\[
P^x_i(\bar{A}(t))=\frac{1}{2\pi i}\int_{|\lambda-\lambda_i(x)|=\frac{\delta_i(x)}{2}}(\lambda-\bar{A}(t))^{-1}d\lambda
\]
depends continuously on $t$, when $|t-x|<\eta_i(x)$ for some small enough $\eta_i(x)>0$, in which $|\lambda-\lambda_i(x)|=\frac{\delta_i(x)}{2}$ is positive oriented.
\end{lemma}

For any fixed $t\in [0,1]$, let $w_i(t)$ be the eigenvector for the eigenvalue $\lambda_i(t)$ of $\bar{A}(t)$, then by the multiplicity-free condition, $\{w_i(t)\}_i$ forms a base for a Hilbert space $H$.\

For any fixed $i$, WOLG, we can assume that $|\lambda_i(t)-\lambda_i(x)|<\frac{\delta_i(x)}{4}$ when $|x-t|<\eta_i(x)$. So, for $t\in (x-\eta_i(x),x+\eta_i(x))$, we have
\begin{multline*}
P^x_i(\bar{A}(t))w_i(t)=\frac{1}{2\pi i}\int_{|\lambda-\lambda_i(x)|=\frac{\delta_i(x)}{2}}(\lambda-\bar{A}(t))^{-1}d\lambda\cdot w_i(t) \\
=\frac{1}{2\pi i}\int_{|\lambda-\lambda_i(x)|=\frac{\delta_i(x)}{2}}(\lambda-\lambda_i(t))^{-1}d\lambda\cdot w_i(t)=\lambda_i(t)w_i(t)
\end{multline*}

\begin{multline*}
P^x_i(\bar{A}(t))w_j(t)=\frac{1}{2\pi i}\int_{|\lambda-\lambda_i(x)|=\frac{\delta_i(x)}{2}}(\lambda-\bar{A}(t))^{-1}d\lambda\cdot w_j(t) \\
=\frac{1}{2\pi i}\int_{|\lambda-\lambda_i(x)|=\frac{\delta_i(x)}{2}}(\lambda-\lambda_j(t))^{-1}d\lambda\cdot w_j(t)=0
\end{multline*}

Since $\{w_i(t)\}_i$ forms a base of $H$, $\forall t\in[0,1]$, then $P_i(\bar{A}_x(t))$ is the continuous $\lambda_i(t)$-eigenprojection-valued function of $\bar{A}(t)$ for $|t-x|<\eta_i(x)$. For $x_1,x_2\in[0,1]$ with $(x_n-\eta_i(x_1),x_1+\eta_i(x_1))\bigcap(x_2-\eta_i(x_2),x_2+\eta_i(x_2))\neq \emptyset$, taking $y\in (x_1-\eta_i(x_1),x_1+\eta_i(x_1))\bigcap(x_2-\eta_i(x_2),x_2+\eta_i(x_2))$, then both $P^{x_1}_i(\bar{A}(y))$ and $P^{x_2}_i(\bar{A}(y))$ are the $\lambda_i(y)$-eigenprojection of $\bar{A}(y)$ and hence $P^{x_1}_i(\bar{A}(y))=P^{x_2}_i(\bar{A}(y))$.\

So for any $i\in\mathbb{N}$, we have a continuous eigenprojection-valued function $P_i(\bar{A}(x))$ s.t.
\[
\bar{A}(x)P_i(\bar{A}(x))=\lambda_i(x)P_i(\bar{A}(x)), ~~~\forall x\in [0,1], \forall i\in\mathbb{N}
\]

Since $P_i(\bar{A}(x))$ defines a 1-dimensional Grassmann bundle over $[0,1]$ and $[0,1]$ is contractible, there exists a continuous section $w_i:[0,1]\to H$ s.t. $||w_i(t)||=1$ and $P_i(\bar{A}(x))w_i(x)=\lambda_i(x)w_i(x)$, $\forall x\in[0,1]$. \

Define a family of unitaries $\{U(x)\}$ by
\[
U(x)e_i=w_i(x), ~~\forall x\in[0,1]
\]
Since $w_i(x)$ is continuous for each $i$, $U(x)$ is a strongly continuous unitary-valued function over $[0,1]$. And so we have $U(x)^*\bar{A}(x)U(x)=\Lambda(x)$, in which
$$\Lambda(x)=\left(
  \begin{array}{ccc}
    \lambda_1(x) & 0 & ... \\
    0 & \lambda_2(x)& ... \\
    ... & ... & ... \\
  \end{array}
\right)$$

Since $\bar{A}$ is the norm-limit of a sequence of matrices over $C[0,1]$, we have the following obvious Lemma.

\begin{lemma}
For the eigenvalue functions $\{\lambda_i(x)\}_{i\in \mathbb{N}}$ and operator $A$ as above, $\forall \epsilon>0$, $\exists N>0$ s.t. $|\lambda_i(x)|<\epsilon$, $||P_n\bar{A}P_n-\bar{A}||<\epsilon$ and $||P_n\bar{B}P_n-\bar{B}||<\epsilon$, $\forall x\in[0,1]$ when $i,n>N$.
\end{lemma}

WOLG, we can also assume $||P_n\Lambda P_n-\Lambda||<\epsilon$, if $n>N$. Therefore for any $n>N$ and $\forall x\in[0,1]$,
\begin{multline*}
||P_nU(x)P_n\Lambda(x) P_n U(x)^* P_n-\bar{A}(x)||\\
\leq ||P_nU(x)P_n\Lambda(x) P_n U(x)^* P_n-P_n\bar{A}(x)P_n||+||P_n\bar{A}(x) P_n-\bar{A}(x)||<2\epsilon
\end{multline*}
and hence we have the following.

\begin{theorem}
For a normal and multiplicity-free element in $\bar{A}\in \mathbb{K}(C[0,1])$ with $Eig(\bar{A}(x))\subset \mathbb{C}\setminus \{0\}$, $\bar{A}$ satisfies condition (1) if and only if it satisfies condition (2).
\end{theorem}

In the following, we would give an example with condition (1) and one without condition (1).

\begin{example}
Let $H$ be a Hilbert space spanned by $\{e_i:i\in\mathbb{Z}\}$.\

Let $\lambda_i(x)\in C[0,1]$, $\forall i\in\mathbb{Z}$ defined by
\begin{eqnarray*}
\begin{cases}
\lambda_n(x)=-\frac{1}{2^{n+1}}x+\frac{1}{2^n}~~~\text{ if }n\geq 0 \cr
\lambda_{-1}(x)=(\frac{3}{2}x-\frac{1}{2})e^{2\pi xi} \cr
\lambda_n(x)=-2^nx-2^n~~~\text{ if }n\leq -2 \end{cases}
\end{eqnarray*}

Let
$\Lambda(x)=\left(
  \begin{array}{ccccc}
    ... & ... & ... & ... & ... \\
    ... & \lambda_1(x) & 0 & 0 & ...\\
    ... & 0 & \lambda_0(x) & 0 & ... \\
    ... & 0 & 0 & \lambda_{-1}(x) & ... \\
   ... & ... & ... & ... & ... \\
  \end{array}
\right)$
and
$U=\left(
  \begin{array}{ccccc}
    ... & ... & ... & ... & ... \\
    ... & 0 & 1 & 0 & ...\\
    ... & 0 & 0 & 1 & ... \\
    ... & 0 & 0 & 0 & ... \\
   ... & ... & ... & ... & ... \\
  \end{array}
\right)$

Let $U(x)$ be the continuous path of unitary s.t. $U(0)=I$ and $U(1)=U$, so $U(0)\Lambda(0)U(0)^*=U(1)\Lambda(1)U(1)^*$. Since unitary does not change the spectrum, $U(x)\Lambda(x)U(x)^*=A(e^{2\pi xi})$ is a normal and multiplicity-free element in $\mathbb{K}(C(S^1))$ with $Eig(A(x))\subset \mathbb{C}\setminus \{0\}$ and has the condition (1).
\end{example}

\begin{example}
Let
$$A(1)=\left(
  \begin{array}{ccccc}
    1 & ... & ... & ... & ... \\
    ... & \frac{1}{2}& 0 & 0 & ...\\
    ... & 0 & \frac{1}{4} & 0 & ... \\
    ... & 0 & 0 & \frac{1}{8} & ... \\
   ... & ... & ... & ... & ... \\
  \end{array}
\right)$$

For any $n\in\mathbb{N}$, let
\begin{eqnarray*}
\begin{cases}
U_n(x)=I,\text{ if }x\notin[\frac{3}{2^{n+2}},\frac{1}{2^n}] \cr
U_n(x)=\left(
  \begin{array}{ccccc}
    I_n & ... & ... & ...  \\
    ... & \cos(2^{n+1}\pi (\frac{1}{2^n}-x))& \sin(2^{n+1}\pi (\frac{1}{2^n}-x)) & ... \\
    ... & -\sin(2^{n+1}\pi (\frac{1}{2^n}-x)) & \cos(2^{n+1}\pi (\frac{1}{2^n}-x)) & ...  \\
   ... & ... & ... & I  \\
  \end{array}
\right)\\
\text{ if }x\in[\frac{3}{2^{n+2}},\frac{1}{2^{n}}]
\end{cases}
\end{eqnarray*}
and
$$U_n=U_n(\frac{3}{2^{n+1}})=\left(
  \begin{array}{ccccc}
    I_n & ... & ... & ...  \\
    ... & 0& 1 & ... \\
    ... & -1 & 0 & ...  \\
   ... & ... & ... & I  \\
  \end{array}
\right)$$

For any $n\in\mathbb{N}$, let
\begin{eqnarray*}
\begin{cases}
V_n(x)=I,\text{ if }x\notin[\frac{1}{2^{n+1}},\frac{3}{2^{n+2}}] \cr
V_n(x)=\left(
  \begin{array}{ccccc}
    I_n & ... & ... & ...  \\
    ... & \frac{1}{\sqrt{2}}(4-2^{n+2}x)e^{2\pi (3-2^{n+2}x)i}& 0& ...  \\
    ... &  & \frac{\sqrt{2}}{4-2^{n+2}x}e^{-2\pi (3-2^{n+2}x)i} & ... \\
    ... & ...& ...  & I  \\
  \end{array}
\right)\\
\text{ if }x\in[\frac{1}{2^{n+1}},\frac{3}{2^{n+2}}]
\end{cases}
\end{eqnarray*}
 and
$$V_n=V_n(\frac{1}{2^{n+1}})=\left(
  \begin{array}{ccccc}
    I_n & ... & ... & ...  \\
    ... & \sqrt{2}& 0 & ... \\
    ... & 0 & \frac{1}{\sqrt{2}} & ...  \\
   ... & ... & ... & I  \\
  \end{array}
\right)$$
Let
\begin{eqnarray*}
\begin{cases}
A(x)=U_n(x)V_{n-1}U_{n-1}...V_0U_0A(1)U_0^*V_0^*...U_{n-1}^*V_{n-1}^*U_n(x)^*, \text{ if }x\in [\frac{3}{2^{n+2}},\frac{1}{2^{n}}]\cr
A(x)=V_n(x)U_n...V_0U_0A(1)U_0^*V_0^*...U_n^*V_n(x)^*, \text{ if }x\in[\frac{1}{2^{n+1}},\frac{3}{2^{n+2}}]\cr
A(0)=A(1)
\end{cases}
\end{eqnarray*}

It is easy to see $A(x)$ is normal multiplicity-free and norm continuous at $x=0$. Moreover $A(x)\in \mathbb{K}(C[0,1])$ and $Eig(A(x))\subset \mathbb{C}\setminus \{0\}$, $\forall x\in [0,1]$. So we have a unit-vector-valued function
\begin{eqnarray*}
w_1(x)=\begin{cases}
U_n(x)V_{n-1}U_{n-1}...V_0U_0e_1, \text{ if }x\in [\frac{3}{2^{n+2}},\frac{1}{2^{n}}]\cr
V_n(x)U_n...V_0U_0e_1, \text{ if }x\in[\frac{1}{2^{n+1}},\frac{3}{2^{n+2}}]
\end{cases}
\end{eqnarray*}
which satisfies the condition $A(x)w_1(x)=\lambda_1(x)w_1(x)$, $\forall x\in (0,1]$. Therefore
\begin{eqnarray*}
|\lambda_1(x)|=\begin{cases}
\frac{1}{2^{n}}, \text{ if }x\in [\frac{3}{2^{n+2}},\frac{1}{2^{n}}]\cr
|2x-\frac{1}{2^{n+1}}|, \text{ if }x\in[\frac{1}{2^{n+1}},\frac{3}{2^{n+2}}]
\end{cases}
\end{eqnarray*}

But if $A(x)$ had the condition (1), then $\lambda_1(x)$ could be continuously extended into an element in $C[0,1]$ with value $\lambda_1(0)=0$. It is a contradiction with the fact $0\notin Eig(A(0))$.

\end{example}

So for two normal multiplicity-free $A$ and $B\in\mathbb{K}(C(S^1))$, as in the argument above Lemma 3.7, there exists a norm continuous diagonal operator-valued function $\bar{\Lambda}(x)$ and strongly continuous unitary-valued functions $U_1(x)$ and $U_2(x)$ s.t.
\[
U_1(x)^*\bar{A}(x)U_1(x)=\bar{\Lambda}(x)=U_2(x)^*\bar{B}(x)U_2(x), ~~\forall x\in[0,1]
\]
in which
$$\bar{\Lambda}(x)=\left(
  \begin{array}{cccc}
   \lambda_1(x) & 0 & ... \\
  0 & \lambda_2(x) & ...\\
     ... & ... & ...  \\
   \end{array}
\right)$$

Since $\bar{A}(0)=\bar{A}(1)$ and hence $\{\lambda_i(0)\}_i=\{\lambda_i(1)\}_i$, there exists a permutation $\sigma:\mathbb{N}\to\mathbb{N}$ s.t. $\lambda_i(0)=\lambda_{\sigma(i)}(1)$,$\forall i\in\mathbb{N}$. By the multiplicity-free condition and computation
\begin{eqnarray*}
\begin{cases}
\bar{A}(0)U_1(0)e_i=U_1(0)\bar{\Lambda}(0)e_i=\lambda_i(0)U_1(0)e_i\cr
\bar{A}(1)U_1(1)e_{\sigma(i)}=U_1(1)\bar{\Lambda}(1)e_{\sigma(i)}=\lambda_{\sigma(i)}(1)U_1(1)e_{\sigma(i)}=\lambda_{i}(0)U_1(1)e_{\sigma(i)}
\end{cases}
\end{eqnarray*}
we can get
\begin{equation}
U_1(0)e_i=U_1(1)e_{\sigma(i)},\forall i\in \mathbb{N}\text{ and similarly }U_2(0)e_i=U_2(1)e_{\sigma(i)},\forall i\in \mathbb{N}
\end{equation}

For any fixed $n\in \mathbb{N}$, let $S_x^n=\{i\in \mathbb{N}:|\lambda_i(x)|\geq \frac{1}{n}\}$ and $S_n=\{i\in\mathbb{N}:\max_{x\in[0,1]}|\lambda_i(x)|\geq \frac{1}{n}\}$.\

By Lemma 3.7, $S_n$ is a finite subset of $\mathbb{N}$.
Let $s(n)=\sharp S_n$ and WLOG we can assume $S_n=[1,s(n)]\bigcap\mathbb{N}$.
For $S_x^n$ at $x=0$ and 1, we have $\sigma(S_0^n)=S_1^n$, but $\sigma(S_n)$ may not equal to $S_n$. So we choose a permutation
$$\sigma':S_n\to S_n$$
s.t.
\begin{equation*}
\sigma'|_{S_n\bigcap\sigma^{-1}(S_n)}=\sigma|_{S_n\bigcap\sigma^{-1}(S_n)}\text{ and }\sigma'(S_n\setminus \sigma^{-1}(S_n))=S_n\setminus \sigma(S_n)
\end{equation*}

 In the following, we will construct an element $\bar{A}_n\in \mathbb{K}(C[0,1])$ s.t.
$$||\bar{A}_n-\bar{A}||<\frac{4}{n}\text{ and }\bar{A}_n(0)=\bar{A}_n(1)$$

$\mathbf{Step 1:}$\\

Since $S_n$ is finite, there exists $\alpha(n)>0$ s.t.
$$|\lambda_i(x)-\lambda_i(1)|<\frac{1}{n},\forall x\in [\alpha(n),1],\forall i\in S_n$$

So if $i\notin S_1^n$, $|\lambda_i(1)|<\frac{1}{n}$ and hence
\begin{equation}
|\lambda_i(x)|<\frac{2}{n},\text{ if }i\in S_n\setminus S_1^n \text{ and }\forall x\in[\alpha(n),1]
\end{equation}

We define a sequence of functions $\lambda_i'(x)$ as follows. Let\
\begin{eqnarray*}
\lambda_i'(x)=\begin{cases}
\lambda_i(x),\text{ if } i\in S_n\bigcap\sigma(S_n)\cr
0,\text{ if } i\notin S_n\cr
\begin{cases}
\lambda_i(x),\text{ for } x\in [0,\alpha(n)]\cr
f_i(x),\text{ for } x\in [\alpha(n),1]
\end{cases}
\text{ if } i\in S_n\setminus \sigma(S_n)
\end{cases}
\end{eqnarray*}
in which $f_i(x)$ satisfies
\[
  f_i(x) =
  \begin{cases}
    f_i(\alpha(n))=\lambda_i(\alpha(n))\\
	f_i(1)=\lambda_{\sigma'^{-1}(i)}(0)\\
    0<|f_i(x)|\leq\max\{|\lambda_i(\alpha(n))|,|\lambda_{\sigma'^{-1}(i)}(0)|\}
  \end{cases}
\]
For any $i\notin S_n\bigcap \sigma(S_n)$ and hence $i\notin S_1^n$ and $\sigma'^{-1}(i)\notin S_0^n$, by (3.2), we have
\begin{equation}
\max_{x\in[\alpha(n),1]}\{|\lambda_i'(x)|\}=\max_{x\in[\alpha(n),1]}\{|f_i(x)|\}
\leq\max\{|\lambda_i(\alpha(n))|,|\lambda_{\sigma'^{-1}(i)}(0)|\}\leq \frac{2}{n}
\end{equation}

Again since $S_n$ is finite, by the multiplicity-free condition, we can assume $f_i(x)\neq f_j(x)$ and $f_i(x)\neq 0$, $\forall i\neq j\in S_n$, $\forall x\in[0,1]$. So we have a new continuous diagonal operator-valued function
$$\bar{\Lambda}'(x)=\left(
  \begin{array}{cccc}
   \lambda_1'(x) & 0 & ... \\
  0 & \lambda_2'(x) & ...\\
     ... & ... & ...  \\
   \end{array}
\right)$$
s.t. $\lambda_i'(x)=0$ when $i\notin S_n$ and $\lambda_i'(x)\neq\lambda_j'(x)$ when $i\neq j\in S_n$, $\forall x\in [0,1]$ and
\begin{equation}
\lambda_i'(0)=\lambda_{\sigma'(i)}'(1),\forall i\in S_n
\end{equation}

$\mathbf{Step 2}$:\\

Let $H(n)=span\{e_i:i\in S_n\}$ and $L(n)=span\{e_i:i\in S_n\bigcup\sigma(S_n)\}$ and let $V_1$ be a partial isometry defined by
$$
 \begin{cases}
    V_1e_i=e_i, \text{ if }i\in S_n\bigcap\sigma(S_n)=\sigma'(S_n\bigcap\sigma^{-1}(S_n))\\
	V_1e_i=e_{\sigma(\sigma'^{-1}(i))},\text{ if }i\in S_n\setminus \sigma(S_n)~i.e.~V_1e_{\sigma'(i)}=e_{\sigma(i)}, \text{ if }\sigma'(i)\in S_n\setminus \sigma(S_n)\\
    V_1e_i=0, \text{ if }i\notin S_n
  \end{cases}
$$
and hence
$$
 \begin{cases}
    V_1^*e_i=e_i, \text{ if }i\in S_n\bigcap\sigma(S_n)=\sigma'(S_n\bigcap\sigma^{-1}(S_n))\\
	V_1^*e_i=e_{\sigma'(\sigma^{-1}(i))},\text{ if }i\in \sigma(S_n)\setminus S_n~i.e.~V_1^*e_{\sigma(i)}=e_{\sigma'(i)}, \text{ if }\sigma'(i)\in S_n\setminus \sigma(S_n)\\
    V_1^*e_i=0, \text{ if }i\notin \sigma(S_n)
  \end{cases}
$$
Let
$$
V_1(x)=\begin{cases}
    P_{H(n)},\text{ if }~~x\in [0,\alpha(n)]\\
	W_1(x), \text{ if }~~x\in (\alpha(n),1)\\
    V_1, \text{ if }~~x=1
  \end{cases}
$$
in which $W_1(x)$ is a continuous path of rank-$s(n)$ partial isometries connecting $P_{H(n)}$ and $V_1$ s.t. $W_1(x)\in P_{L(n)}B(H)P_{L(n)}$ and $W_1(x)e_i=e_i$, $\forall i\in S_n\bigcap \sigma(S_n)$. Therefore $V_1(x)e_i=e_i$, $\forall i\in S_n\bigcap \sigma(S_n)$ and
\begin{equation*}
V_1(x)\in P_{L(n)}B(H)P_{L(n)}
\end{equation*}

Since $V_1(x)^*e_i\in ran\{e_i:i\notin S_n\bigcap\sigma(S_n)\}$ , $\forall i\notin S_n\bigcap\sigma(S_n)$, by (3.3)
\begin{multline}
\max_{x\in[\alpha(n),1]}\{||V_1(x)\bar{\Lambda}'(x)V_1^*(x)e_i||:i\notin S_n\bigcap \sigma(S_n)\}\\
\leq\max_{x\in[\alpha(n),1]}\{|\lambda'_i(x)|:i\notin S_n\bigcap \sigma(S_n)\}\leq \frac{2}{n}
\end{multline}

By the definition of $V_1(x)$, we have

$
\begin{cases}
    U_1(0)V_1(0)\bar{\Lambda}'(0)V_1(0)^*U_1(0)^*U_1(0)e_i=\lambda'_i(0)U_1(0)e_i,\forall i\in S_n\cr
    U_1(0)V_1(0)\bar{\Lambda}'(0)V_1(0)^*U_1(0)^*U_1(0)e_i=0,\forall i\notin S_n
  \end{cases}
$

and by (3.1),(3.4) and the fact $\sigma(i)=\sigma'(i),\forall i\in S_n\bigcap \sigma(S_n)$
\begin{align*}
&U_1(1)V_1(1)\bar{\Lambda}'(1)V_1(1)^*U_1(1)^*U_1(0)e_i\\
&=U_1(1)V_1(1)\bar{\Lambda}'(1)V_1(1)^*e_{\sigma(i)}\\
&=U_1(1)V_1(1)\bar{\Lambda}'(1)e_{\sigma'(i)}\\
&=U_1(1)V_1(1)\lambda'_{\sigma'(i)}(1)e_{\sigma'(i)}\\
&=U_1(1)\lambda'_{\sigma'(i)}(1)e_{\sigma(i)}\\
&=\lambda'_{i}(0)U_1(0)e_i,~~~\forall i\in S_n
\end{align*}
and
\[
U_1(1)V_1(1)\bar{\Lambda}'(1)V_1(1)^*U_1(1)^*U_1(0)e_i=0,\forall i\notin S_n
\]
Let $\bar{A}_n(x)=U_1(x)V_1(x)\bar{\Lambda}'(x)V_1(x)^*U_1(x)^*\in \mathbb{K}(C[0,1])$ and we have $\bar{A}_n(0)=\bar{A}_n(1),\forall i\in S_n$.\\

$\mathbf{Step 3}$:\\

By the definition of $V_1(x)$ and $S_n,S_n^1,S_n^0$ and (3.2),(3.4),(3.5), we have the following estimations:\

If $i\notin S_n\bigcup\sigma(S_n)$, $||V_1(x)\bar{\Lambda}'(x)V_1(x)^*e_i-\bar{\Lambda}(x)e_i||=||\bar{\Lambda}(x)e_i||=|\lambda_i(x)|<\frac{1}{n}$.\

$
\text{If }i\in \sigma(S_n)\setminus S_n\begin{cases}
    ||V_1(x)\bar{\Lambda}'(x)V_1(x)^*e_i-\bar{\Lambda}(x)e_i||
    \leq||\bar{\Lambda}(x)e_i||=|\lambda_i(x)|<\frac{1}{n}\cr
    \text{ when } x\in [0,\alpha(n)] \cr
||V_1(x)\bar{\Lambda}'(x)V_1(x)^*e_i-\bar{\Lambda}(x)e_i||\cr
\leq||V_1(x)\bar{\Lambda}'(x)V_1(x)^*e_i||+|\lambda_i(x)|<\frac{2}{n}+\frac{2}{n}=\frac{4}{n}\cr
\text{ when }x\in [\alpha(n),1]
  \end{cases}
$\

$
\text{If }i\in S_n\bigcap\sigma(S_n),~~||V_1(x)\bar{\Lambda}'(x)V_1(x)^*e_i-\bar{\Lambda}(x)e_i||
    \leq|\lambda_i'(x)-\lambda_i(x)|=0
$\

$
\text{If }i\in S_n\setminus \sigma(S_n)\begin{cases}
    ||V_1(x)\bar{\Lambda}'(x)V_1(x)^*e_i-\bar{\Lambda}(x)e_i||\leq|\lambda_i'(x)-\lambda_i(x)|=0\cr
    \text{when } x\in [0,\alpha(n)]\cr
||V_1(x)\bar{\Lambda}'(x)V_1(x)^*e_i-\bar{\Lambda}(x)e_i||\cr
    \leq||V_1(x)\bar{\Lambda}'(x)V_1(x)^*e_i||+|\lambda_i(x)|\leq \frac{2}{n}+\frac{2}{n}=\frac{4}{n}\cr
    \text{when }x\in [\alpha(n),1]
  \end{cases}
$

Therefore we have
$$||V_1(x)\bar{\Lambda}'(x)V_1(x)^*-\bar{\Lambda}(x)||<\frac{4}{n}$$
and hence
$$||U_1(x)V_1(x)\bar{\Lambda}'(x)V_1(x)^*U_1(x)^*-\bar{A}(x)||<\frac{4}{n}$$
So for $\bar{A}_n(x)=U_1(x)V_1(x)\bar{\Lambda}'(x)V_1(x)^*U_1(x)^*$,
\begin{equation}
||\bar{A}_n(x)-\bar{A}(x)||<\frac{4}{n},~~\forall x\in[0,1]
\end{equation}

By similar argument, for $U_2(x)\bar{\Lambda}(x)U_2(x)^*=\bar{B}(x)$, $\forall x\in[0,1]$, we can find a continuous family of partial isometries $V_2(x)$, $x\in[0,1]$ s.t.
$$\bar{B}_n=U_2(x)V_2(x)\bar{\Lambda}'(x)V_2(x)^*U_2(x)^*,~~~~||\bar{B}_n-\bar{B}||<\frac{4}{n}\text{ and }\bar{B}_n(0)=\bar{B}_n(1)$$

Since, by the definition of $V_1(x)$, $V_1(x)V_1(x)e_i=e_i$, $\forall i\in S_n$, then\

$\begin{cases}
  \bar{\Lambda}'(x)e_i=\lambda_i(x)e_i,~\forall i\in S_n\cr
  \bar{\Lambda}'(x)e_i=0,~\forall i\notin S_n
  \end{cases}$\
  
implies\

$\begin{cases}
  V_1(x)\bar{\Lambda}'(x)V_1(x)^* V_1(x)e_i=\lambda_i(x) V_1(x)e_i,~\forall i\in S_n\cr
  V_1(x)\bar{\Lambda}'(x)V_1(x)^*V_1(x)e_i=0,~\forall i\notin S_n
  \end{cases}$\

And we have
\begin{align*}
&\{\lambda'_i(x)\}_{i\in S_n}=Eig^{\neq 0}(\bar{\Lambda}'(x))=Eig^{\neq 0}(V_1(x)\bar{\Lambda}'(x)V_1(x)^*)\\
&=Eig^{\neq 0}(U_1(x)V_1(x)\bar{\Lambda}'(x)V_1(x)^*U_1(x)^*),~~\forall i\in S_n
\end{align*}

Because conjugation by unitaries keeps eigenvalues, we have
\begin{align*}
&\{\lambda'_i(x)\}_{i\in S_n}=Eig^{\neq 0}(\bar{A}_n(x))=\{\lambda\in Eig(\bar{A}_n(x)):\lambda\neq 0\}\\
&=Eig^{\neq 0}(\bar{B}_n(x))=\{\lambda\in Eig(\bar{B}_n(x)):\lambda\neq 0\}
\end{align*}

Since the nonzero eigenvalues vary continuously as functions, the eigenprojections, corresponding to nonzero eigenvalues, $\{\bar{p}_i(x)\}_{i\in S_n}$ of $\bar{A}_n$ and $\{\bar{q}_i(x)\}_{i\in S_n}$ of $\bar{B}_n$ also vary continuously, in which $\{\bar{p}_i(0)\}_{i\in S_n}=\{\bar{p}_i(1)\}_{i\in S_n}$ and $\{\bar{q}_i(0)\}_{i\in S_n}=\{\bar{q}_i(1)\}_{i\in S_n}$. So we have
\[
\bar{P}_{S_n}(0)\doteq \sum_{i\in S_n}\bar{p}_i(0)=\sum_{i\in S_n}\bar{p}_i(1)\doteq\bar{P}_{S_n}(1)
\]
\[
\bar{Q}_{S_n}(0)\doteq \sum_{i\in S_n}\bar{q}_i(0)=\sum_{i\in S_n}\bar{q}_i(1)\doteq\bar{Q}_{S_n}(1)
\]

Since $\bar{A}_n(0)=\bar{A}_n(1)$ and $\bar{B}_n(0)=\bar{B}_n(1)$, then we have $A_n$, $B_n\in\mathbb{K}(C(S^1))$ defined by $A_n(e^{2\pi xi})=\bar{A}_n(x)$ and $B_n(e^{2\pi xi})=\bar{B}_n(x)$, $\forall x\in [0,1]$. By (3.6) and its analogue for $B$, we have
\begin{equation}
||A-A_n||<\frac{4}{n}\text{ and }||B-B_n||<\frac{4}{n}
\end{equation}

Since $\{\lambda_i'(0)\}_{i\in S_n}=\{\lambda_i'(1)\}_{i\in S_n}$, we have
$$Eig^{\neq0}(B_n(e^{2\pi xi}))=Eig^{\neq0}(A_n(e^{2\pi xi}))=\{\lambda_i'(x)\}_{i\in S_n},~\forall x\in [0,1]$$
and
\begin{equation}
\sharp Eig^{\neq0}(A_n(t))=\sharp Eig^{\neq0}(B_n(t))=s(n),~\forall t\in S^1
\end{equation}

Let
\[
\{p_i(e^{2\pi xi})\}_{i\in S_n}=\{\bar{p}_i(x)\}_{i\in S_n}\text{ and }\{q_i(e^{2\pi xi})\}_{i\in S_n}=\{\bar{q}_i(x)\}_{i\in S_n},~~\forall x\in[0,1]
\]
and
\[
P_{S_n}(t)=\sum_{i\in S_n}p_i(t) \text{ and }Q_{S_n}(t)=\sum_{i\in S_n}q_i(t),~~\forall t\in S^1
\]

So, $\forall t\in S^1$, $\{p_i(t)\}_{i\in S_n}$, $\{q_i(t)\}_{i\in S_n}$ are the eigenprojection sets, corresponding to nonzero eigenvalues, of $A_n$ and $B_n$, both of which, vary continuously as a set. By the same reason, $P_{S_n}(t)$ and $Q_{S_n}(t)$ are continuous rank-$s(n)$ projection-valued operators for $t\in S^1$. Similar with the argument above Theorem 2.7, there is a continuous map $\Phi^{s(n)}_{A_n,B_n}$ from $S^1$ to the base $C_{s(n)}$ of fiber bundle $E$.

\begin{theorem}
Let $A,B$ be two normal multiplicity-free elements in $\mathbb{K}(C(S^1))$ with $Eig(A(t))=Eig(B(t))\subset \mathbb{C}\setminus \{0\}$, $\forall t\in S^1$, which satisfy the condition (1). If the map $\Phi_{A_n,B_n}^{s(n)}:S^1\to C_{s(n)}$ lifts to a continuous map $\tilde{\Phi}_{A_n,B_n}^{s(n)}:S^1\to E_{s(n)}$, then  there exists an operator $U\in I+\mathbb{K}(C(S^1))$ s.t.
$$||U(t)A(t)U(t)^*-B(t)||<\frac{37}{n},\forall t\in S^1\text{ and }U^*U=I=UU^*$$
\end{theorem}
\begin{proof}
Suppose $\pi \tilde{\Phi}_{A_n,B_n}^{s(n)}=\Phi_{A_n,B_n}^{s(n)}$ for some continuous map $\tilde{\Phi}_{A_n,B_n}^{s(n)}:S^1\to E_{s(n)}$. So for any $ t\in S^1$,
$$\tilde{\Phi}_{A_n,B_n}^{s(n)}(t)=(\Phi_{A_n,B_n}^{s(n)}(t),W_{s(n)}(t))$$
in which $W_{s(n)}(t)$ is a norm continuous rank-$s(n)$ isometry on $S^1$. Similar with the argument in the proof of Theorem 2.7, $\forall t\in S^1$
\begin{equation*}
 \begin{cases}
W_{s(n)}^*(t)W_{s(n)}(t)=P_{S_n}(t)\\
W_{s(n)}(t)W_{s(n)}^*(t)=Q_{S_n}(t)\\
W_{s(n)}(t)A_n(t)W_{s(n)}^*(t)=B_n(t)
  \end{cases}
\end{equation*}
and hence
\begin{equation*}
 \begin{cases}
1-W_{s(n)}^*(t)W_{s(n)}(t)=(1-W_{s(n)}^*(t)W_{s(n)}(t))^{\frac{1}{2}}\\
W_{s(n)}(t)A_n(t)=B_n(t)W_{s(n)}(t)\\
A_n(t)W_{s(n)}^*(t)=W_{s(n)}^*(t)B_n(t)
 \end{cases}
\end{equation*}
By a standard argument on compact space, we have $\forall n\in \mathbb{N}$, $\exists m(n)\in \mathbb{N}$ with $P_{m(n)}e_i=0,~\forall i\geq m(n)+1$ and $P_{m(n)}e_i=e_i,~\forall i\leq m(n)$ s.t. $\forall t\in S^1$
\begin{equation*}
 \begin{cases}
||P_{m(n)}W_{s(n)}(t)P_{m(n)}A_n(t)P_{m(n)}W_{s(n)^*}(t)P_{m(n)}-P_{m(n)}B_n(t)P_{m(n)}||<\frac{1}{n}\\
||P_{m(n)}A_n(t)P_{m(n)}-A_n(t)||<\frac{1}{n}\\
||P_{m(n)}B_n(t)P_{m(n)}-B_n(t)||<\frac{1}{n}\\
||(1-P_{m(n)}W_{s(n)}^*(t)P_{m(n)}W_{s(n)}(t)P_{m(n)})\\
-(1-P_{m(n)}W_{s(n)}^*(t)P_{m(n)}W_{s(n)}(t)P_{m(n)})^{\frac{1}{2}}||<\frac{1}{n(||A||+||B||)}\\
||P_{m(n)}W_{s(n)}(t)P_{m(n)}A_n(t)P_{m(n)}\\
-P_{m(n)}W_{s(n)}(t)P_{m(n)}A_n(t)P_{m(n)}W_{s(n)}^*(t)P_{m(n)}W_{s(n)}(t)P_{m(n)}||<\frac{1}{n}\\
||P_{m(n)}A_n(t)P_{m(n)}W_{s(n)}^*(t)P_{m(n)}\\
-P_{m(n)}W_{s(n)}^*(t)P_{m(n)}W_{s(n)}(t)P_{m(n)}A_n(t)P_{m(n)}W_{s(n)}^*(t)P_{m(n)}||<\frac{1}{n}
 \end{cases}
\end{equation*}
and hence
\begin{equation}
 \begin{cases}
||P_{m(n)}W_{s(n)}(t)P_{m(n)}A(t)P_{m(n)}W_{s(n)^*}(t)P_{m(n)}-P_{m(n)}B(t)P_{m(n)}||<\frac{9}{n}\\
||P_{m(n)}A(t)P_{m(n)}-A(t)||<\frac{9}{n}\\
||P_{m(n)}B(t)P_{m(n)}-B(t)||<\frac{9}{n}\\
||(1-P_{m(n)}W_{s(n)}^*(t)P_{m(n)}W_{s(n)}(t)P_{m(n)})\\
-(1-P_{m(n)}W_{s(n)}^*(t)P_{m(n)}W_{s(n)}(t)P_{m(n)})^{\frac{1}{2}}||<\frac{1}{n(||A||+||B||)}\\
||P_{m(n)}W_{s(n)}(t)P_{m(n)}A(t)P_{m(n)}\\
-P_{m(n)}W_{s(n)}(t)P_{m(n)}A(t)P_{m(n)}W_{s(n)}^*(t)P_{m(n)}W_{s(n)}(t)P_{m(n)}||<\frac{9}{n}\\
||P_{m(n)}A(t)P_{m(n)}W_{s(n)}^*(t)P_{m(n)}\\
-P_{m(n)}W_{s(n)}^*(t)P_{m(n)}W_{s(n)}(t)P_{m(n)}A(t)P_{m(n)}W_{s(n)}^*(t)P_{m(n)}||<\frac{9}{n}
 \end{cases}
\end{equation}

Let $U_{m(n)}'(t)=P_{m(n)}W_{s(n)}(t)P_{m(n)}$ for any $t\in S^1$, the last 3 inequalities in (3.9) are replaced by
\begin{equation*}
 \begin{cases}
||(I_{m(n)}-U_{m(n)}'^*(t)U_{m(n)}'(t))-(I_{m(n)}-U_{m(n)}'^*(t)U_{m(n)}'(t))^{\frac{1}{2}}||<\frac{1}{n(||A||+||B||)}\\
||U_{m(n)}'(t)P_{m(n)}A(t)P_{m(n)}(I_{m(n)}-U_{m(n)}'^*(t)U_{m(n)}'(t))||<\frac{9}{n}\\
||(I_{m(n)}-U_{m(n)}'^*(t)U_{m(n)}'(t))P_{m(n)}A(t)P_{m(n)}U_{m(n)}'(t)||<\frac{9}{n}
 \end{cases}
\end{equation*}
which induces $\forall t\in S^1$
\begin{equation*}
 \begin{cases}
||U_{m(n)}'(t)P_{m(n)}A(t)P_{m(n)}(I_{m(n)}-U_{m(n)}'^*(t)U_{m(n)}'(t))^{\frac{1}{2}}||<\frac{10}{n}\\
||(I_{m(n)}-U_{m(n)}'^*(t)U_{m(n)}'(t))^{\frac{1}{2}}P_{m(n)}A(t)P_{m(n)}U_{m(n)}'(t)||<\frac{10}{n}
 \end{cases}
\end{equation*}
By the exercise 9.3 in \cite{Rordam},
$$U_{2m(n)}''(t)=\left(
  \begin{array}{cc}
    U'_{m(n)}(t) & (I_{m(n)}-U'_{m(n)}(t)U'_{m(n)}(t)^*)^{\frac{1}{2}} \\
    -(I_{m(n)}-U'_{m(n)}(t)^*U'_{m(n)}(t))^{\frac{1}{2}} & U'_{m(n)}(t)^* \\
  \end{array}
\right)$$
is a unitary for any $t\in S^1$ and then
\begin{align*}
&||U_{2m(n)}''(t)P_{m(n)}A(t)P_{m(n)}U_{2m(n)}''^*(t)-P_{m(n)}B(t)P_{m(n)}||\\
&\leq \frac{10}{n}+||P_{m(n)}U_{2m(n)}''(t)P_{m(n)}A(t)P_{m(n)}U_{2m(n)}''^*(t)P_{m(n)}-P_{m(n)}B(t)P_{m(n)}||\\
&=\frac{10}{n}+||U_{m(n)}'(t)A(t)U_{m(n)}'^*(t)-P_{m(n)}B(t)P_{m(n)}||\leq \frac{19}{n}
\end{align*}
Let
$$U(t)=\left(
  \begin{array}{cc}
    U'_{2m(n)}(t) & 0 \\
    0 & I \\
  \end{array}
\right)\in I+\mathbb{K}(C(S^1))$$
and by the second and third inequalities in  (3.9), we have, for any $t\in S^1$
\begin{align*}
&||U(t)A(t)U^*(t)-B(t)||\\
&\leq ||U(t)P_{m(n)}A(t)P_{m(n)}U^*(t)-P_{m(n)}B(t)P_{m(n)}||+\frac{18}{n}\\
&=||U_{2m(n)}''(t)P_{m(n)}A(t)P_{m(n)}U_{2m(n)}''^*(t)-P_{m(n)}B(t)P_{m(n)}||+\frac{18}{n}\leq \frac{37}{n}
\end{align*}
\end{proof}

By the same argument in section 3.2 of \cite{Park}, let $\Phi_{A_n,B_n}^*E_{s(n)}$ be the pullback bundle of $\Phi_{A_n,B_n}:S^1\to E_{s(n)}$ over $S^1$. Because the fibers of  $E_{s(n)}$ are homeomorphic to $\mathbb{T}^{s(n)}$, so are the fibers $F_x$ of $\Phi_{A_n,B_n}^*E_{s(n)}$ and $\pi_1(F_x)\cong \mathbb{Z}^{s(n)}$. Therefore we have a bundle of groups $\pi_1(F_x)$, which would be denoted as $\Pi_{A_n,B_n}$. As $S^1$ can be treated as a cell complex with no cells of dimension greater than 1, similar with the argument in section 3.2 of \cite{Park}, we have:

\begin{corollary}
Let $A,B$ be two normal multiplicity-free elements in $\mathbb{K}(C(S^1))$ with $Eig(A(t))=Eig(B(t))\subset \mathbb{C}\setminus \{0\}$, $\forall t\in S^1$, if $A,B$ satisfy the condition (1), then $A$ and $B$ are approximately unitarily equivalent.
\end{corollary}

For $A,B$ be two normal multiplicity-free elements in $\mathbb{K}(C(S^1))$ with $Eig(A(t))=Eig(B(t))\subset \mathbb{C}\setminus \{0\}$, $\forall t\in S^1$, since the condition (1) ensures the approximately unitary equivalence, next question is what happen without condition (1)? Firstly, we need a lemma, which ensures that any nonzero eigenvalue $\lambda_{x_0}\in Eig(\bar{A}(x_0))$ can be extended locally to be a continuous function $\lambda_{x_0}(x)\in C[x_0-a_{x_0},x_0+a_{x_0}]$ s.t. $\lambda_{x_0}(x)\in Eig(\bar{A}(x))$, $\forall x\in[x_0-a_{x_0},x_0+a_{x_0}]$ and $\lambda_{x_0}(x_0)=\lambda_{x_0}$, for some $a_{x_0}>0$, in which $\bar{A}(x)=A(e^{2\pi xi})$ for $x\in [0,1]$. \

For any fixed ${x_0}\in[0,1]$, $\lambda_i(x_0)\in Eig(\bar{A}(x_0))\subset\mathbb{C}\setminus \{0\}$, since $|\lambda_i(x_0)|>0$ and zero is the only accumulating point of $Eig(\bar{A}(x_0))$, there exists $\delta_i(x_0)>0$ s.t.
$$B(\lambda_i({x_0}),\delta_i({x_0}))\bigcap B(\lambda_j({x_0}),\delta_j({x_0}))=\emptyset~\forall i\neq j\in\mathbb{N}$$
$$B(\lambda_i({x_0}),\delta_i({x_0}))\bigcap B(0,\delta_i({x_0}))=\emptyset~\forall i\in\mathbb{N}$$

\begin{lemma}
For any fixed $x_0\in [0,1]$ and any $\bar{A}(x)$, $\lambda_i(x_0)$ and $\delta_i(x_0)$ in Lemma 3.6, $\forall 0<\beta<\delta_i(x_0)$, $\exists \alpha_{i,\beta}(x_0)>0$, s.t.
\[\sharp(B(\lambda_i(x_0),\frac{\beta}{4})\bigcap Eig(\bar{A}(x)))=1,~~x\in(x_0-\alpha_{i,\beta}(x_0),x_0+\alpha_{i,\beta}(x_0))\]
\end{lemma}
\begin{proof}
Because $\overline{B(\lambda_i(x_0),\frac{\beta}{2})}\setminus B(\lambda_i(x_0),\frac{\beta}{4})$ is a closed subset and $(\lambda-\bar{A}(x_0))^{-1}$ exists  $\forall \lambda\in\overline{B(\lambda_i(x_0),\frac{\beta}{2})}\setminus B(\lambda_i(x_0),\frac{\beta}{4})$, we have
\[
\inf\{||\lambda-\bar{A}(x_0)^{-1}||:\lambda\in\overline{B(\lambda_i(x_0),\frac{\beta}{2})}\setminus B(\lambda_i(x_0),\frac{\beta}{4})\}=r_{i,\beta}(x_0)>0
\]
So
\begin{equation*}\
 (\lambda-A(x))^{-1}=(\lambda-A(x_0))^{-1}(1-(A(x)-A(x_0))(\lambda-A(x_0))^{-1})^{-1}
\end{equation*}
exists for $||\bar{A}(x)-\bar{A}(x_0)||<r_{i,\beta}(x_0)$ , $\forall \lambda\in\overline{B(\lambda_i(x_0),\frac{\beta}{2})}\setminus B(\lambda_i(x_0),\frac{\beta}{4})$.\

Since $\bar{A}(x)$ is continuous, $(\lambda-A(x))^{-1}$ exists for $x\in(x_0-\alpha_{i,\beta}(x_0),x_0+\alpha_{i,\beta}(x_0))$ for some $\alpha_{i,\beta}(x_0)>0$ and $\forall \lambda\in\overline{B(\lambda_i(x_0),\frac{\beta}{2})}\setminus B(\lambda_i(x_0),\frac{\beta}{4})$. Therefore
\begin{equation}
(\overline{B(\lambda_i(x_0),\frac{\beta}{2})}\setminus B(\lambda_i(x_0),\frac{\beta}{4}))\bigcap Eig(\bar{A}(x))=\emptyset,\forall x\in(x_0-\alpha_{i,\beta}(x_0),x_0+\alpha_{i,\beta}(x_0))
\end{equation}
Let $f$ be a continuous function defined on $\mathbb{C}$ satisfying the condition
\[
1.f|_{\overline{B(\lambda_i(x_0),\frac{\beta}{4})}}=1,~~~~~2.f|_{B(\lambda_i(x_0),\frac{\beta}{2})^c}=0,~~~~~3.0\leq f\leq 1
\]

So $f(\bar{A}(x))$ is a continuous path of projections for $x\in(x_0-\alpha_{i,\beta}(x_0),x_0+\alpha_{i,\beta}(x_0))$ and $f(\bar{A}(x_0))$ is a rank one projection. Therefore $f(\bar{A}(x))$ is a rank one projection, $\forall x\in (x_0-\alpha_{i,\beta}(x_0),x_0+\alpha_{i,\beta}(x_0))$ i.e.
\begin{multline*}
  \exists \alpha_{i,\beta}(x_0)>0\text{, s.t. }\\
  \sharp(B(\lambda_i(x_0),\frac{\beta}{4})\bigcap Eig(\bar{A}(x)))=1,\forall x\in(x_0-\alpha_{i,\beta}(x_0)),x_0+\alpha_{i,\beta}(x_0)))
\end{multline*}

\end{proof}

So for a normal multiplicity-free $\bar{A}(x)\in \mathbb{K}(C[0,1])$ with nonzero eigenvalues at each $x$ but without condition (1), the lemma above ensures the locally continuous extension of the nonzero eigenvalues. Then, for $\bar{A}\in \mathbb{K}(C[0,1])$ without condition (1), we can assume that $\exists \lambda_0\in Eig(\bar{A}(0))$ s.t. its locally continuous extension $\lambda_0(x)$ satisfies the following condition
$$\lim_{t\to \alpha}\lambda_0(t)\text{ does not exist or }\lim_{t\to \alpha}\lambda_0(t)=0$$
in which
$$\alpha=\sup\{\gamma:\exists\lambda_0(t)\in C[0,\gamma]~s.t.\lambda_0(0)=\lambda_0\text{ and }\lambda_0(t)\in Eig(\bar{A}(t)), \forall t\in [0,\gamma]\}\leq 1$$

Since $||\bar{A}||<\infty$, $\lim_{t\to\alpha}\lambda_0(t)$ can not be infinite. So if $\lim_{t\to \alpha^-}\lambda_0(t)$ does not exist and we can assume that there exists an increasing sequence $\{x_k\}_k(x_k\to \alpha^-\text{ as }k\to\infty)$ s.t.
\[
\lim_{n\to\infty}\lambda_0(x_{2n})=a\neq b=\lim_{n\to\infty}\lambda_0(x_{2n+1})
\]

WOLG, we can assume $a\neq 0$. If $a\notin Eig(\bar{A}(\alpha))$, we can find $(\frac{|a|}{2}>)\epsilon>0$ s.t.
\[
B(a,\epsilon)\bigcap (N_{\epsilon}(Eig(\bar{A}(\alpha)))\bigcup B(0,\epsilon))=\emptyset
\]
in which $N_{\epsilon}(E)=\{x:d(x,E)<\epsilon\}$, then
$$Eig(\bar{A}(x_{2n}))\ni\lambda_0(x_{2n})\notin N_{\epsilon}(Eig(\bar{A}(\alpha)))\bigcup B(0,\epsilon)$$
for $n$ big enough, which is a contradiction with Lemma 3 in page 80 of \cite{Baumgartel}. So we have $a\in Eig(\bar{A}(\alpha))$ and by Lemma 3 in page 80 of \cite{Baumgartel}, \

$\exists(\min\{\frac{|a|}{2},\frac{|a-b|}{3}\}>)\delta>0$ and $\eta_{\delta}>0$ s.t.
\[
(\overline{B(a,\delta)}\setminus B(a,\frac{\delta}{2}))\bigcap Eig(\bar{A}(x))=\emptyset,~~\forall x\in[\alpha-\eta_{\delta},\alpha]
\]
which is a contradiction with the existence of a path of value in $Eig(\bar{A}(x))$ from $\lambda_0(x_{2n})$ to $\lambda_0(x_{2n+1})$ for $n$ big enough.\

\begin{proposition}
For a normal multiplicity-free $\bar{A}(x)\in \mathbb{K}(C[0,1])$ with condition $Eig(\bar{A}(x))\subset \mathbb{C}\setminus \{0\}$, which does not satisfy condition (1), we can assume $\exists \lambda_0\in Eig(\bar{A}(0))$ s.t. its corresponding locally continuous extension $\lambda_0(x)\in C[0,\alpha)$ satisfies the condition $\lim_{t\to \alpha}\lambda_0(t)=0$.
\end{proposition}

\begin{theorem}
Let $A,B$ be two normal multiplicity-free elements in $\mathbb{K}(C(S^1))$ with $Eig(A(x))=Eig(B(x))\subset \mathbb{C}\setminus \{0\}$, $\forall x\in S^1$, if $A$ satisfied the condition (1) so does $B$. Similarly for two normal multiplicity-free elements $\bar{A},\bar{B}$ in $\mathbb{K}(C[0,1])$ with $Eig(\bar{A}(x))=Eig(\bar{B}(x))\subset \mathbb{C}\setminus \{0\}$, $\forall x\in [0,1]$, if $\bar{A}$ satisfied the condition (1), so does $\bar{B}$.
\end{theorem}
\begin{proof}
We only prove for $\bar{A}$ and $\bar{B}$. If not, we have two normal multiplicity-free elements $\bar{A},\bar{B}$ in $\mathbb{K}(C[0,1])$ with $Eig(\bar{A}(x))=Eig(\bar{B}(x))\subset \mathbb{C}\setminus \{0\}$, $\forall x\in [0,1]$ and $\bar{A}$ satisfies the condition (1) while $\bar{B}$ does not. We can assume that $\exists\alpha\in(0,1], \exists\lambda_0\in Eig(\bar{B}(0))$ s.t. its corresponding locally continuous extension $\lambda_0(x)\in C[0,\alpha)$ satisfies the condition $\lim_{x\to \alpha}\lambda_0(x)=0$.\

Since $Eig(\bar{A}(x))=Eig(\bar{B}(x))$, there exists $\mu_0\in C[0,1]$ s.t. $\mu_0(0)=\lambda_0$ and $\mu_0(x)\in Eig(\bar{A}(x))$, $\forall x\in[0,1]$. Because $\mu_0(\alpha)\neq 0$, then $b=\sup\{a\in[0,1]:\mu_0(a)=\lambda_0(a)\}<\alpha$ and by the continuity of $\lambda_0(x)$ and $\mu_0(x)$, $\mu_0(b)=\lambda_0(b)>0$. Since $Eig(\bar{A}(x))=Eig(\bar{B}(x))$, $\forall x\in[0,1]$, so $\forall \epsilon>0$, $\exists $ a decreasing sequence $\{x_k\}_k$ with $x_k>b$ and $\lim_{k\to\infty}x_k=b$ s.t. for $k$ big enough, we have $\mu_0(x_k)\neq\lambda_0(x_k)$ and
\[
\mu_0(x_k),\lambda_0(x_k)\in B(\mu_0(b),\epsilon)\bigcap Eig(\bar{A}(x_k))=B(\mu_0(b),\epsilon)\bigcap Eig(\bar{B}(x_k))
\]
which is a contradiction with Lemma 3.13.
\end{proof}

\begin{corollary}
Let $A,B$ be two normal multiplicity-free elements in $\mathbb{K}(C(S^1))$ with $Eig(A(x))=Eig(B(x))\subset \mathbb{C}\setminus \{0\}$, $\forall x\in S^1$, if $A$ satisfies the condition (1), then $A$ and $B$ are approximately unitarily equivalent.
\end{corollary}

$\mathbf{Question}$: For two normal multiplicity-free elements $A,B$ in $\mathbb{K}(C(S^1))$ with $Eig(A(x))=Eig(B(x))\subset \mathbb{C}\setminus \{0\}$, $\forall x\in S^1$ without condition (1),  whether they are approximately unitarily equivalent? Or whether there exists a pair of normal multiplicity-free elements $A,B$ in $\mathbb{K}(C(S^1))$ with $Eig(A(x))=Eig(B(x))\subset \mathbb{C}\setminus \{0\}$, which neither satisfy condition (1) while $A$ and $B$ are approximately unitarily equivalent?\\

{\bf Acknowledgments} The author wishes to thank the reviewers for careful
reading and valuable comments. This work was supported by NSFC grant of P.R. China (No.11326104). And the authors are
grateful to Moscow State University where part of this paper has been written.

\bibliographystyle{amsplain}

\end{document}